\documentclass[fleqn,reqno]{amsart}
\usepackage[a4paper,twoside=false]{geometry}
\usepackage{lmodern} 
\usepackage[defaultsans]{lato}
\usepackage[utf8]{inputenc}
\usepackage{mathtools}
\usepackage{amsmath,amssymb,amsfonts}
\usepackage{stmaryrd} 
\usepackage{enumitem}
\usepackage{microtype} 
\usepackage[dvipsnames]{xcolor} 
\usepackage[colorlinks, allcolors = blue]{hyperref}
\usepackage{varioref} 
\usepackage{leading} 
\leading{14pt} 
\usepackage{tikz}
\usetikzlibrary{cd,babel}
\usetikzlibrary{matrix,calc,arrows}

\makeatletter
\AtBeginDocument{%
\expandafter\let\csname[\endcsname\relax
\expandafter\let\csname]\endcsname\relax
\expandafter\DeclareRobustCommand\csname[\expandafter\endcsname\expandafter{%
  \csname begin\endcsname{equation}%
}%
\expandafter\DeclareRobustCommand\csname]\expandafter\endcsname\expandafter{%
  \csname end\endcsname{equation}%
  }%
}
\mathtoolsset{showonlyrefs,showmanualtags}

\usepackage[initials, alphabetic,  msc-links]{amsrefs}

\newtheorem{Theorem}{Theorem}[section]
\newtheorem{Definition}[Theorem]{Definition}
\newtheorem{Proposition}[Theorem]{Proposition}
\newtheorem{Lemma}[Theorem]{Lemma}
\newtheorem{Corollary}[Theorem]{Corollary}
\newtheorem{TheoremIntro}{Theorem}

\theoremstyle{remark}
\newtheorem{Example}[Theorem]{Example}
\newtheorem{Remark}[Theorem]{Remark}
\newlist{thmlist}{enumerate}{1}
\setlist[thmlist]{
        nolistsep,
        ref={\mdseries\textup{(\emph{\roman*})}},
        label={\mdseries\textup{(\emph{\roman*})}},
        }

\newlist{tfae}{enumerate}{1}
\setlist[tfae]{
        nolistsep,
        ref={\mdseries\textup{(\emph{\alph*})}},
        label={\mdseries\textup{(\emph{\alph*})}},
        before={\advance\mathindent\leftmargin}}
\makeatletter
\newcommand\tfaeitem[1]{{\textup{(\emph{\@alph #1})}}}

\makeatother

\newcommand\claim[1]{%
  \begin{minipage}{.8\displaywidth}%
  \itshape
  #1
  \end{minipage}%
}

\DeclareMathOperator{\im}{Im}

\DeclareMathOperator{\Ext}{Ext}
\DeclareMathOperator{\Hom}{Hom}
\DeclareMathOperator{\Der}{Der}
\DeclareMathOperator{\Aut}{Aut}
\DeclareMathOperator{\End}{End}
\DeclareMathOperator{\Diff}{Diff}
\DeclareMathOperator{\coker}{coker}
\DeclareMathOperator{\Gr}{Gr}
\DeclareMathOperator{\GL}{GL}

\DeclarePairedDelimiter\abs{\lvert}{\rvert}

\DeclarePairedDelimiter\lin{\langle}{\rangle}

\newcommand\nset[1]{\llbracket#1\rrbracket}

\newcommand\NN{\mathbb{N}}

\newcommand\CC{\mathbb{C}}
\renewcommand\epsilon{\varepsilon}

\let\*\bullet

\renewcommand\SS{\mathfrak{S}}
\newcommand\HH{H\!H}
\newcommand\kk{\Bbbk}

\newcommand\X{\mathfrak{X}}
\newcommand\D{\Diff}
\newcommand\A{\mathcal{A}}
\renewcommand\AA{\mathbb A}
\newcommand\B{\mathcal{B}}

\newcommand\inc{\hookrightarrow}

\newcommand\xx{\hat{x}}

\DeclareMathOperator{\Out}{OutDer}

\setitemize{nolistsep, listparindent=\parindent}

\title[The Hochschild cohomology~$H(S,U)$]{
Center and Lie algebra of outer derivations for
algebras of differential
operators associated to hyperplane arrangements}

\author{Francisco Kordon}
\author{Thierry Lambre}
\address{
CONICET
and
Instituto Balseiro, Universidad Nacional de Cuyo – CNEA.
Av.\,Bustillo 9500, San Carlos de Bariloche, R8402AGP,
R\'io Negro,
Argentina}
\email{franciscokordon@gmail.com}
\address{
Laboratoire de Mathématiques Blaise Pascal,
UMR6620 CNRS,
Université Clermont Auvergne,
Campus des Cézeaux,
3 place Vasarely,
63178 Aubière cedex,
France}
\email{thierry.lambre@uca.fr}

\date{\today}

\begin{document}

\begin{abstract}
We compute the center and the Lie algebra of outer derivations of a familiy
of algebras of differential operators associated to hyperplane arrangements
of the affine space $\mathbb A^3$. The results are completed for $4$-braid
arrangements and for reflection arrangements associated to the wreath product
of a cyclic group with the symmetric group $\SS_3$. To
achieve this we use tools from homological algebra and Lie--Rinehart algebras
of differential operators.
\end{abstract}

\maketitle


\section*{Introduction}

Let $V$ be a finite-dimensional $\kk$-vector space over a field $\kk$ of
characteristic zero, $S$ the algebra of coordinates on $V$ and $\A$ a central
hyperplane arrangement in $V$.
We assume throughout the article that $\A$ is a free arrangement
in the sense given by
K.\,Saito in~\cite{saito}: we require that the
Lie algebra~$\Der\A$ of derivations of $S$ tangent to $\A$ is a free
$S$-module.
%
The algebra $\Diff\A$ of differential operators tangent
to~$\A$,
as seen by F.\,J.\,Calderon Moreno in~\cite{calderon} and by 
M.\,Suárez--Álvarez
in~\cite{differential-arrangements},
is the subalgebra of~$\End(S)$ generated by $\Der\A$ and $S$.
Our results concern the center and the Lie algebra of
outer derivations of $\Diff\A$.

The first and simplest example of a free arrangement is the case
of a central line arrangement in $V=\kk^2$. This case is studied by the first
author and M.\,Suárez-Álvarez in ~\cite{ksa} when there are at least $5$
lines: a description of the Hochschild cohomology~$\HH^\*\left( \Diff\A
\right)$, including its cup product and Gerstenhaber bracket, is given
explicitly in detail through a calculation independent of the methods that we
now use.
The second and most important family of
examples is that of the braid arrangements $\B_n$, given for $n\geq2$ by the
hyperplanes $H_{ij}=\{x\in\kk^n :x_i=x_j\}$ with $i\neq j$: these arrangements
are free and have served historically as a proxy to obtain general
results, for instance in V.\,I.\,Arnold's classical article~\cite{arnold}.

In virtue of the freeness of $\A$ the algebra $\Diff\A$ is isomorphic to the
enveloping algebra of a Lie--Rinehart algebra~$(S,L)$ ---see L.\,Narvaez
Macarro's~\cite{narvaez} and the first author's thesis~\cite{tesis}--- and
then the spectral sequence introduced by both authors in~\cite{kola} permits
the computation of~$\HH^\bullet(\Diff\A)$ in terms of the Hochschild
cohomology $H^\*(S,\Diff\A)$ of~$S$ with values on $\Diff\A$ and the
Lie--Rinehart cohomology of~$L$.  This was successfully applied to
arrangements of three lines in~\cite{kola}, and, ultimately, to $\A = \B_3$
---see Corollary~\ref{coro:hhB3}.

The homological approach described above allows us to compute the center of
$\Diff\A$ under the hypothesis that the Saito's matrix of the arrangement~$\A$
is triangular: more generally, we can state this result resorting to the
hypothesis of triangularizability of Lie--Rinehart algebras that we give in
Definition~\ref{def:conditions}.

\begin{TheoremIntro}[Theorem~\ref{thm:hh0}]
Let $(S,L)$ be a triangularizable Lie--Rinehart algebra with enveloping
algebra $U$. The center of $U$ is $\kk$.
\end{TheoremIntro}

Let $\A_r$, $r\geq1$, be the arrangement in $\CC^3$ defined by 
$0 = xyz(z^r-y^r)(z^r-x^r)(y^r-x^r)$.
This arrangement is~$\B_4$ when $r=1$. When $r\geq2$, it is the
reflection arrangement of 
the wreath product of the cyclic group of order $r$ and the symmetric
group~$\SS_3$.
The homological method yields the following result.


\begin{TheoremIntro}[Corollary~\ref{coro:abelian}]
Let $r\geq1$.
For each hyperplane $H$ in $\A_r$ let $f_H$ be a linear form with kernel $H$
and
$\partial_H$ the derivation of $\Diff\A_r$ determined by 
\[
\begin{cases*}
  \partial_H(g) = 0 & if $g\in \kk[x_1,x_2,x_3]$; \\
  \partial_H(\theta) = \theta(f_H)/f_H & if $\theta\in\Der\A_r$.
\end{cases*}
\]
The Lie algebra of outer derivations of $\Diff\A_r$ together with the
commutator is an abelian Lie algebra of dimension~$3r+3$,
the numbers of
hyperplanes of $\A_r$, and is generated by the
classes of the derivations $\partial_H$ with $H\in\A_r$.
\end{TheoremIntro}

In the pursuit of $\HH^\*(U)$ a key step is the computation of $H^\*(S,U)$.
We succeeded in its calculation when $\*=0,1$ for a family of Lie--Rinehart
algebras that generalizes $\Der\A_r$.
The  result in Corollary~\ref{coro:H1:result:graded} relates $H^1(S,U)$ to
the cokernel of the Saito's matrix --- this is an important object of the theory
with a rich algebraic structure studied, for instance, 
by M.\,Granger, D.\,Mond and M.\,Schulze
in~\cite{schulze}. 

There are several ways in which the calculations performed in this article can
be continued. In particular, following the methods of J.\,Alev and
M.\,Chamarie in~\cite{AC} our findings on the algebra of outer derivations of
$\Diff\A$ can lead to a description of $\Aut(\Diff\A)$ as in \cite{ksa}*{\S7}
and M.\,Suárez-Álvarez and Q.\,Vivas'~\cite{SAV}.

\bigskip

The first author is currently a CONICET postdoctoral fellow and received
support from  BID PICT 2019-00099.  We thank the Universit\'e Clermont
Auvergne for hosting the first author in a postdoctoral position at the
Laboratoire de Math\'ematiques Blaise Pascal during the year
2019-2020.

\bigskip

Unadorned $\Hom$ and $\End$ are taken over $\kk$. The set of natural numbers
~$\NN$ is that of nonnegative integers.
If $n$ and $m$ are positive integers, we denote by $\nset{n,m}$ the set of
integers $k$ such that $n\leq k\leq m$, and 
$\nset{m} \coloneqq \nset{1,m}$.

\section{Generalities}
\subsection{Hyperplane arrangements}

\begin{Definition}
A central \emph{hyperplane arrangement} $\A$ in a finite dimensional vector
space $V$ is a finite set $\{H_1,\ldots,H_\ell\}$ of subspaces of codimension
1.
Choosing a basis of~$V$ we may identify the algebra $S(V^*)$ of coordinates
of~$V$ with $S=\kk[x_1,\ldots,x_n]$: for each $i\in\nset{\ell}$ let
$\lambda_i\in S$ be a linear form with kernel $H_i$. Up to a nonzero scalar,
the \emph{defining polynomial} $Q=\lambda_1\cdots\lambda_\ell\in S $ depends
only on~$\A$.
\end{Definition}

\begin{Definition}
The set of
\emph{derivations tangent to the arrangement} $\A$ is
\[\label{eq:arrangements:derivations}
  \Der \A \coloneqq 
  \left\{ \theta\in\Der(S)~:
    \text{$\lambda_i$ divides $\theta (\lambda_i)$ for every $i\in\nset{\ell}$
}
  \right\}.
\] 
This is a Lie--subalgebra and a sub-$S$-module of the Lie algebra of
derivations $\Der(S)$ of $S$.
\end{Definition}

\begin{Definition}
The arrangement $\A$ is \emph{free} if $\Der\A$ is a free $S$-module. 
\end{Definition}

\begin{Theorem}[Saito's criterion,~\cite{saito}*{Theorem 1.8.ii}]
\label{thm:saito}
A family of $\ell$ derivations $ (\theta_1,\ldots,\theta_\ell )$ in $\Der\A$
is an $S$-basis of $\Der\A$ 
if and only if the determinant of \emph{Saito's matrix}
\[\label{eq:saitomatrix}
  M(\theta_1,\ldots,\theta_\ell)
  \coloneqq
  \begin{pmatrix}
		\theta_1(x_1) & \cdots & \theta_1(x_\ell) \\
		\vdots & ~ & \vdots \\
		\theta_\ell(x_1) & \cdots & \theta_\ell(x_\ell) 
  \end{pmatrix}
\]
is a nonzero scalar multiple of $Q$.
\end{Theorem}

The notion of freeness connects arrangements of hyperplanes with commutative
algebra, algebraic geometry and combinatorics.  While not a property of generic
hyperplane arrangements, many of the motivating
examples of hyperplane arrangements are free.  Saito’s criterion in
Theorem~\ref{thm:saito} is perhaps the most practical way to prove freeness,
though there are other methods to prove this condition ---see
A.\,Bigatti, E.\,Palezzato and M.\,Torielli's~\cite{bigatti} for a discussion
on the state of the art.



\begin{Example}\label{ex:braids:notilde}
Let $n\geq2$, $E=\AA^n$ the affine space with coordinate ring
$S=\kk[x_1,\ldots,x_n]$.
The braid arrangement $\B_n$ in $E$ has hyperplanes $H_{ij}$ with equation
$x_i-x_j = 0$, $1\leq i < j \leq n$, so that the defining polynomial
is $Q = \prod_{1\leq i < j \leq n}\left( x_j-x_i \right)$. 

Consider the derivations $\theta_1,\ldots,\theta_n$ of $S$ defined for 
$k\in\nset{n}$ by
\begin{align}\label{eq:braids:notilde:basis}
  & 
  \theta_{1}(x_k) = 1,
  &&
  \theta_i(x_k) 
  = (x_k-x_1)\ldots(x_k-x_{i-1})
  \qquad\text{if $i\geq2$}.
\end{align}
These derivations satisfy $(x_k-x_j) \mid \theta_i(x_k-x_j)$ for any
$i,j,k\in\nset{n}$ and therefore belong to $\Der\B_{n}$.  The Saito's matrix
$\left( \theta_i(x_k) \right)$ is triangular and its determinant is 
$Q$. By Saito's Criterion, $\Der\B_{n}$ is a free $S$-module with
basis $\left\{ \theta_1,\ldots,\theta_n \right\}$.
\end{Example}

\begin{Example}\label{ex:braids:tilde}
Let $n\geq1$ and $E=\AA^{n+1}$ be the affine space with coordinate ring
$S=\kk[x_0,x_1,\ldots,x_n]$.
As in Example~\ref{ex:braids:notilde} above, the arrangement $\B_{n+1}$
in $E$ has equation $\prod_{0\leq i < j\leq n}(x_i-x_j)$.

Consider the subspace $V=\left\{ 0 \right\}\times\AA^n$ of $E$, defined by
the equation $ x_0 = 0$, and the hyperplanes $\tilde H_{ij}$ of $V$ defined by
$x_i-x_j=0$ for $1\leq i < j \leq n$.  We call $\tilde\B_n$ the arrangement
formed by these hyperplanes, so that $\tilde\B_n$ is defined by equation
$x_1\ldots x_n\prod_{0\leq i < j\leq n}(x_i-x_j)=0$.  The derivations
$\alpha_1,\ldots,\alpha_n$ of
$S=\kk[x_1,\ldots,x_n]$ defined for $k\in\nset{n}$~by 
\begin{align*}
  \alpha_1(x_k) = x_k,
  &&
  \alpha_i(x_k) 
  = 
  \begin{cases*}
    0 & if $i>k$; \\
    x_k\prod_{j< i}(x_k-x_j) & if $i\leq k$
  \end{cases*}
  \quad\text{if $i\geq2$}
\end{align*} 
belong to $\Der\tilde\B_n$. Thanks to Saito's Criterion,
$(\alpha_1,\ldots,\alpha_n)$ is a basis of $\Der\tilde\B_n$.
\end{Example}

\begin{Example}
Let $V$ be a finite-dimensional vector space. We say that $\sigma\in\GL(V)$ is
a pseudo-reflection if $\sigma$ is of finite order and fixes a hyperplane
$H_\sigma$ of $V$, and it is a reflection if this order is~$2$.
A finite subgroup $G$ of the group of automorphisms of $V$
is a \emph{(pseudo-) reflection group} if it is generated by (pseudo-)
reflections, and the set of reflecting
hyperplanes~$\A(G)$ of a reflection group $G$ is the \emph{reflection
arrangement} of~$G$. It is a result by H\@. Terao in~\cite{terao} that
every
reflection arrangement over $\kk = \CC$ is free.

Consider the $n$th braid
arrangement $\B_n$ of Example~\ref{ex:braids:notilde}: identifying the
reflection with respect to the plane $x_i-x_j=0$ with the permutation
$(ij)\in\SS_n$ we see that $\B_n=\A(\SS_n)$.
\end{Example}

\begin{Example}\label{ex:wreath}
Let $r,n\geq1$ and consider the arrangement $\A_r^n$ in $V=\kk^n$ defined by
\[\label{eq:Q:wreath} 0 = x_1\ldots x_n\prod_{1\leq i<j\leq n}(x_j^r-x_i^r).
\] Taking $r=1$ we see that $\A_1^n=\tilde\B_n$ for every $n$.  When $r\geq2$,
let $G=C_r\wr\SS_n$ be the wreath product of the cyclic group $C_r$ of order
$r$ and the symmetric group $\SS_n$. We see that $\A_r^n$ is the reflection
arrangement of $G$, this is, $\A_r^n = \A(C_r\wr\SS_n)$.
There is a well-known basis of $\Der\A_r^n$ in \cite{OT}*{\S B} that consists of
the derivations
$\theta_1,\ldots,\theta_n$ of $S=\kk[x_1,\ldots,x_n]$ defined
for $1\leq k,m\leq n$ by
\(
  \theta_m(x_k) = x_k^{(m-1)r+1}. 
\)
Consider the derivations $\alpha_1,\ldots,\alpha_n$ of $S$
defined for $1\leq k\leq n$ and $2\leq m\leq n$ by
\begin{align}\label{eq:symmetric}
  & \alpha_1(x_k) = x_k,
  &&\alpha_m(x_k) = x_k\prod_{i=1}^{m-1}(x_k^r - x_i^r).
\end{align}
These derivations belong to $\Der\A_r^n$: evidently
$\alpha_1=\theta_1$, and  if $m\geq2$ then
\[
  \alpha_m = \theta_m - s_1\theta_{m-1} + \ldots +(-1)^{m-1}s_{m-1}\theta_1,
\]
where $s_j = \sum_{1\leq i_1<\ldots<i_{j}\leq m-1}x_{i_1}^r\cdots x_{i_j}^r$
is the $j$th elementary symmetric polynomial in variables
$x_1^r,\ldots,x_{m-1}^r$ for $1\leq j\leq m-1$.
For $1\leq k\leq n$
\begin{align*}
\MoveEqLeft
  \left( \theta_m - s_1\theta_{m-1} + \ldots +(-1)^{m-1}s_{m-1}\theta_1 \right)
    (x_k)
  \\
  & = x_k^{(m-1)r+1} - s_1x_k^{(m-2)r+1} + \ldots + (-1)^{m-1}s_{m-1}x^k
  = x_k\prod_{i=1}^{m-1}(x_k^r - x_i^r),
\end{align*}
which equals $\alpha_m(x_k)$.
Saito's matrix $\left(\alpha_m(x_k) \right)$ is diagonal and its
determinant is 
\[
  \prod_{k=1}^n\alpha_k(x_k) 
  = \prod_{k=1}^n x_k\prod_{i=1}^{k-1}(x_k^r - x_i^r)
  = x_1\ldots x_n\prod_{1\leq i<k\leq n}(x_k^r-x_i^r).
\]
It follows from Saito's criterion that $\alpha_1,\ldots,\alpha_n$ is a
basis of $\Der\A_r^n$. 
\end{Example}

\begin{Example}\label{ex:Ar}
Let $r\geq1$. 
The arrangement $\A_r\coloneqq\A_r^3$ is defined by the nullity of
\[
    Q(\A_r) \coloneqq x_1x_2x_3(x_2^r-x_1^r)(x_3^r-x_1^r)(x_3^r-x_2^r).
\]
The basis of $\Der\A_{r}$ in~\eqref{eq:symmetric} consist in this case of the
derivations $\alpha_1,\alpha_2,\alpha_3$ of $S= \kk[x_1,x_2,x_3]$ with Saito's
matrix
\[
  \begin{pmatrix}
    x_1 & x_2 & x_3 \\
    0 & x_2(x_2^r-x_1^r) & x_3(x_3^r-x_1^r) \\
    0 & 0 & x_3(x_3^r-x_2^r)(x_3^r-x_1^r)
  \end{pmatrix}
\] 
\end{Example}

\subsection{Lie--Rinehart algebras}\label{subsec:l-r}

\begin{Definition}
Let $S$ and $(L,[-,-])$ be, respectively, a commutative and a Lie algebra
endowed with a morphism of Lie algebras $L\to\Der(S)$ that we write
$\alpha\mapsto\alpha_S$  and  a left $S$-module structure on $L$ which we
simply denote by juxtaposition. We say that the pair $(S,L)$ is a
\emph{Lie--Rinehart algebra}, or that $L$ is a Lie--Rinehart algebra over $S$, 
if the equalities
\begin{align}
  &(s\alpha )_S(t) = s\alpha_S(t),
  &&[\alpha,s\beta] = s[\alpha,\beta] + \alpha_S(s)\beta
\end{align} 
hold whenever $s,t\in S$ and $\alpha, \beta \in L$.
\end{Definition}

If $S$ is a commutative algebra and $L$ is a Lie-subalgebra of the Lie algebra
of derivations $\Der S $ that is at the same time an $S$-submodule then $L$ is
an Lie--Rinehart algebra over $S$. This applies to our situation of interest:

\begin{Proposition}
Let $\A$ be a hyperplane arrangement in a vector space $V$. 
The Lie algebra of derivations $\Der\A$ of $\A$
is a Lie--Rinehart algebra over the algebra of coordinates of $V$.
\end{Proposition}

\begin{Definition}\label{def:conditions}
Let $n\geq1$, $S=\kk[x_1,\ldots,x_n]$ and $L$ be a subset  of
derivations of $S$ such that $(S,L)$ is a Lie-Rinehart algebra.
\begin{thmlist}
\item
We call $L$ \emph{triangularizable} if 
$L$ is a free $S$-module that admits a basis  given by derivations
$\alpha_1,\ldots,\alpha_n$ satisfying  the two conditions
\begin{align*}
  &\alpha_i(x_j) = 0 \quad\text{if $i>j$,}
  &&\alpha_1(x_1)\cdots\alpha_n(x_n)\neq0. 
\end{align*}

\item
We say that $L$ satisfies the \emph{Bézout condition} if in addition
for each  $k$ in $\nset{n-1}$, the element $\alpha_k(x_k)$ of $S$ is coprime
with the determinant of the matrix 
\[
  \begin{pmatrix}
    \alpha_k(x_{k+1}) & \alpha_{k+1}(x_{k+1}) & 0 & \cdots & 0 
    \\
    \vdots & \vdots & \ddots & \ddots & \vdots
    \\
    \alpha_k(x_{n-2}) & \alpha_{k+1}(x_{n-2}) & \cdots & \alpha_{n-2}(x_{n-2}) 
    & 0
    \\ 
    \alpha_{k}(x_{n-1}) & \alpha_{k+1}(x_{n-1}) &\cdots &\cdots 
    & \alpha_{n-1}(x_{n-1}) 
    \\
    \alpha_{k}(x_{n}) &\alpha_{k+1}(x_{n}) & \cdots& \cdots & \alpha_{n-1}(x_{n})
  \end{pmatrix}
\]
\end{thmlist}
\end{Definition}

\begin{Example}\label{ex:braids:2conditions}
For any $n\geq2$ the Lie--Rinehart algebra $\Der\B_n$ is triangular and
satisfies the Bézout condition with the basis~$\left\{
\theta_1,\ldots,\theta_n \right\}$ given in Example~\ref{ex:braids:notilde}.
The same goes to $\Der\tilde\B_n$ with the basis given in
Example~\ref{ex:braids:tilde}.
\end{Example}

\begin{Example}\label{ex:wreath:triangular}
Let $r,n\geq1$. 
The Lie--Rinehart algebra associated to $\A(C_r\wr\SS_n)$ is triangularizable,
as follows immediately from Example~\ref{ex:wreath}.
\end{Example}

\begin{Example}
Let $r\geq1$. The arrangement $\A_r=\A(C_r\wr\SS_3)$ from Example~\ref{ex:Ar} 
is triangularizable thanks to Example~\ref{ex:wreath}.
Moreover, it satisfies the Bézout condition: 
indeed, $\alpha_2(x_2) = x_2 (x_2^r-x_1^r)$
is coprime with $\alpha_2(x_3) = x_3 (x_3^r-x_1^r)$, and the determinant
\[
  \det
  \begin{pmatrix}
    \alpha_1(x_2) & \alpha_1(x_3) \\
    \alpha_2(x_2) & \alpha_2(x_3)
  \end{pmatrix}
  =
  \det
  \begin{pmatrix}
    x_2  & x_3  \\
    x_2 (x_2^r-x_1^r) & x_3 (x_3^r-x_1^r)
  \end{pmatrix}
  = x_2  x_3  \left( x_3^r-x_2^r \right)
\]
is coprime with 
$\alpha_1(x_1) = x_1 $.
\end{Example}

\subsection{Differential operators associated to an arrangement}

Remember from J.\,C.\,McConnell and J.\,C.\,Robson's~\cite{MR}*{\S 15} that
the algebra~$\Diff S$ of differential operators on $S=\kk[x_1,\ldots,x_n]$ is
the subalgebra of $\End
S$ generated by $\Der S$ and the set of maps given by left multiplication by
elements of $S$.
Recall as well from~\cite{MR}*{\S5} that 
if $R$ is an algebra and $I\subset
R$ is a right ideal, the largest subalgebra $\mathbb I_R(I)$ of $R$ that
contains $I$ as an ideal ---the \emph{idealizer} of $I$ in $R$--- 
is~$\{r\in R:rI\subset I\}$.

\begin{Definition}
Let $\A$ be a central arrangement of hyperplanes with defining polynomial~$Q$.
The \emph{algebra of differential
operators tangent to the arrangement}~$\A$ is
\[
\Diff(\A)
	= \bigcap_{n\geq1} \mathbb I_{\Diff(S)}(Q^n\Diff (S)).
\]
\end{Definition}

As seen in \cite{calderon} for $\kk=\CC$ or
in~\cite{differential-arrangements} for $\kk$ of characteristic zero, if $\A$
is free then the algebra~$\Diff\A$ coincides with the sub-associative algebra
of~$\End(S)$ generated by $\Der\A$ and the set of maps given by left
multiplication by elements of $S$.

\begin{Example}
The arrangement $\A = \tilde\B_2$ in $\kk^2$
with equation $0=xy(y-x)$ admits, by~\cite{kola}*{\S5}, a presentation
of $\Diff\A$ adapted from~\cite{ksa}: the two derivations
\begin{align*}
  &E=x\partial_x + y\partial_y,
  && D = y(y-x)\partial_y
\end{align*}
of $\kk[x,y]$ form a basis of $\Der\A$, and 
the algebra~$\Diff\A$ is generated by the symbols $x$, $y$, $D$ and
$E$ subject to the relations
\begin{align}
  & [y,x] = 0, \\
  & [D,x] = 0, && [D,y] = y(y-x), \\
  & [E,x] = x, && [E,y] = y, && [E,D] = D.
\end{align}
\end{Example}


Given a Lie-Rinehart algebra $(S,L)$, a \emph{Lie--Rinehart module} ---or
$(S,L)$-module--- is a vector space $M$ which is at the same time an
$S$-module and an $L$-Lie module in such a way that
if $s\in S$, $\alpha\in L$ and $m\in M$ 
then
\begin{align}\label{eq:L-Rmodule}
  &\left( s\alpha \right)\cdot m = s\cdot(\alpha\cdot m), 
  &&\alpha\cdot (s\cdot m) = (s\alpha)\cdot m + \alpha_S(s)\cdot m.
\end{align}  

\begin{Theorem}[\cite{hueb}*{\S1}]
Let $(S,L)$ be a Lie-Rinehart algebra.
\begin{thmlist}
\item There exists an associative algebra $U=U(S,L)$, the \emph{universal
enveloping algebra of $(S,L)$}, endowed with a morphism of algebras $i:S\to U$
and a morphism of Lie algebras $j:L\to U$ satisfying, for $s\in S$ and
$\alpha\in L$,
\begin{align}\label{eq:universal}
  &i(s)j(\alpha) 	= j(s\alpha),
  && j(\alpha)i(s) -i(s)j(\alpha) = i(\alpha_S(s)). 
\end{align}

\item The algebra $U$ is universal with these properties. 

\item The category of $U$-modules is isomorphic to the category
of $(S,L)$-modules.
\end{thmlist}
\end{Theorem}

\begin{Example}\label{ex:weyl}
If $S = \kk[x_1 , \ldots , x_n ]$ then the full Lie algebra of derivations $L =
\Der S$ is a Lie–Rinehart algebra and its enveloping algebra is isomorphic to
the algebra of differential operators $\Diff(S) = A_n$, the $n$th Weyl algebra.
\end{Example}

The following result ---\cite{narvaez}*{\S12} when $\kk=\CC$
and~\cite{tesis}*{Theorem 2.19} for $\kk$ of characteristic zero--- is our
motivation to consider Lie--Rinehart algebras in the algebraic aspects of
hyperplane arrangements.

\begin{Theorem}
\label{thm:diff-LR}
Let $\A$ be a free hyperplane arrangement on a vector space~$V$ and let $S$ be
the algebra of coordinate functions on $V$. There is a canonical isomorphism
of algebras
\[
	U(S,\Der\A)\cong\D(\A).
\]
\end{Theorem}

\begin{Proposition}\label{prop:braidshh}
For $n\geq1$ there is an isomorphism of algebras
\[\label{eq:diffBn}
  \Diff \B_{n+1} \cong A_1 \otimes \Diff\tilde\B_n.
\]
\end{Proposition}

\begin{proof}
Let $n\in\NN$, $S= \kk[x_0,x_1,\ldots,x_n]$, $T= \kk[y_1,\ldots,y_n]$ and
observe that the unique morphism of algebras $\kk[z]\otimes T\to S$ given by
$z\mapsto x_0$ and $y_k\mapsto x_k-x_0$ if $k\geq1$ is an
isomorphism ---we are identifying $z$ with $z\otimes 1$ and
$y_k$ with $1\otimes y_k$. 

The derivations in the
basis of $\Der\B_{n+1}$ given in Example~\ref{ex:braids:notilde} induce
derivations
$\tilde\theta_1,\ldots,\tilde\theta_{n+1}$ on $\kk[z]\otimes T$.
 For
$1\leq i\leq n+1$ and $1\leq k \leq n$ these derivations satisfy
\begin{align*}
  & \tilde \theta_{1} : 
    \begin{cases*}
      z\mapsto 1 ; \\
      y_k \mapsto 0 ;
    \end{cases*}
  && \tilde \theta_{2} : 
  \begin{cases*}
      z\mapsto 0 ; \\
      y_k \mapsto y_k;
    \end{cases*}
  && \tilde \theta_{i} : 
  \begin{cases*}
      z\mapsto 0 ; \\
      y_k \mapsto y_k  \prod_{j=1}^{i-1}(y_k-y_j)   
    \end{cases*}
  \quad\text{if $i\geq3$.}
\end{align*}
The Lie algebra $\Der S$ is isomorphic to the Lie algebra product
$\Der\kk[z]\times \Der\tilde\B_n$: the derivations $\tilde\theta_i$ with
$i\geq2$ correspond to the $\alpha_i$'s in Example~\ref{ex:braids:tilde}.
It follows that the enveloping algebra of the Lie--Rinehart pair
$(S,\Der\B_{n+1})$ is isomorphic to the product
$U(\kk[z],\Der\kk[z])\times U(T,\Der\tilde\B_n)$. The result is
now a consequence of Theorem~\ref{thm:diff-LR} and Example~\ref{ex:weyl}.
\end{proof}

\begin{Definition}\label{def:ort}
Let $n\geq1$, $ S=\kk[x_1,\ldots,x_n]$ and $L$ a triangularizable
Lie--Rinehart algebra over $S$ with basis $(\alpha_1,\ldots,\alpha_n)$.
We say that $L$ satisfies
the \emph{orthogonality condition} if there exists a family
$(u_1,\ldots,u_n)$ of elements of $U$ and $f_k^i\in S$ for $1\leq k \leq n$
and
$1\leq i\leq n-1$ such that
\begin{align*}
  & u_k = \alpha_n + \sum_{i=1}^{n-1}f_k^i\alpha_i,
  && [u_k,x_l]=0    \quad\text{if $k\neq l$.} 
\end{align*}
\end{Definition}

\begin{Example}\label{ex:braids:orthogonality}
Consider for $n\geq2$ the Lie--Rinehart algebra $\Der\B_n$ from
Example~\ref{ex:braids:notilde}. 
The family $( u_1,\ldots,u_n )$ of elements of~$U$ defined for
$k\in\nset{n}$ by
\[
  u_k 
  = \sum_{i=k}^n(-1)^{n-i}\prod_{j=i+1}^n\left( x_j-x_k \right)\theta_i
\]
is such that $[u_k,x_l]=0$, whence the orthogonality condition is
satisfied. The Lie--Rinehart algebra $\Der\tilde\B_n$ from
Example~\ref{ex:braids:tilde} also satisfies this condition with
a similar choice of orthogonal~elements.
\end{Example}

Let $r\geq1$ and $\A_r=\A(C_r\wr\SS_3)$.  Let
$S=\kk[x_1,x_2,x_3]$ and $L=\Der\A_{r}$ be the Lie-Rinehart algebra associated
to $\A_{r}$. The derivations $\left\{ \alpha_1,\alpha_2,\alpha_3
\right\}$ given in Example~\ref{ex:Ar}
make of $L$ a triangular Lie algebra that satisfies the Bézout condition.
We identify the universal enveloping algebra of $L$ with $\Diff\A_r$.

\begin{Proposition}\label{prop:Ar:basis}
The Lie--Rinehart algebra associated to $\A_r$ together with the family
$\left\{ u_1,u_2,u_3\right\}$ of elements of $\Diff\A_r$ defined by
\begin{align*}
  & u_1 = \alpha_3 -(x_3^r-x_1^r)\alpha_2  
          + (x_3^r-x_1^r)(x_2^r-x_1^r)\alpha_1,
  && u_2 = \alpha_3 - ( x_3^r-x_2^r )\alpha_2,
  && u_3 = \alpha_3
\end{align*}
satisfies the orthogonality condition.
\end{Proposition}

\begin{proof}
The condition $[u_k,x_l]=0$ if $k,l\in\nset{3}$ and $l\neq k$
holds true whenever
$l<k$, so we suppose that $l>k$. If $k=3$ there is nothing to see; the case
$k=2$ amounts to the verification that
\begin{align*}
  [u_2,x_3]  
  &
  =\alpha_3(x_3) - (x_3^r-x_2^r)\alpha_2(x_3)
  \\
  &= x_3(x_3^r-x_2^r)(x_3^r-x_1^r)- ( x_3^r-x_2^r ) x_3(x_3^r-x_1^r)
  = 0
\end{align*}
and for or $k=1$ we have
\begin{align*}
  [u_1,x_2] 
  &
  = \alpha_3(x_2) -(x_3^r-x_1^r)\alpha_2(x_2)  
    + (x_3^r-x_1^r)(x_2^r-x_1^r)\alpha_1(x_2)
  \\
  &
  = -(x_3^r-x_1^r)x_2(x_2^r-x_1^r)  
    + (x_3^r-x_1^r)(x_2^r-x_1^r)x_2
  = 0
\shortintertext{and}
  [u_1,x_3] 
  &
  = \alpha_3(x_3) -(x_3^r-x_1^r)\alpha_2(x_3)  
    + (x_3^r-x_1^r)(x_2^r-x_1^r)\alpha_1(x_3)
  \\
  & 
  = x_3(x_3^r-x_2^r)(x_3^r-x_2^r)
    -(x_3^r-x_1^r)x_3(x_3^r-x_1^r)  
    + (x_3^r-x_1^r)(x_2^r-x_1^r)x_3
  = 0.
\end{align*}
\end{proof}

\subsection{Cohomology}
Given an associative algebra $A$ the (associative) enveloping algebra $A^e$ is
the vector space $A\otimes A$ endowed with the product~$\cdot$ defined by
$\left( a_1\otimes a_2 \right)\cdot \left( b_1\otimes b_2  \right)= a_1b_1\otimes b_2a_2$, so that the
category of left $A^e$-modules is equivalent to that of $A$-bimodules. The
\emph{Hochschild cohomology} of $A$ with values on an $A^e$-module $M$ is
\[
  H^\*(A,M)\coloneqq \Ext_{A^e}^\*(A,M).
\]
When $M=A$ we write $\HH^\*(A)\coloneqq H^\*(A,M)$. C.\,Weibel's
book~\cite{weibel} may serve as general reference on this subject.

\begin{Definition}[\cite{rinehart}]\label{def:cohom:l-r}
Let $(S,L)$ be a Lie--Rinehart algebra with enveloping algebra $U$ and let $N$
be an $U$-module. The \emph{Lie--Rinehart cohomology of~$(S,L)$ with values on
$N$} is
\[
  H_S^\*(L,N)\coloneqq\Ext^\*_{U}(S,N).
\]
\end{Definition}

\begin{Remark}[\cite{rinehart}]\label{rk:ch-ei}
In the setting of Definition~\ref{def:cohom:l-r} above, suppose that $L$ is
$S$-projective and let $\Lambda_S^\*L$ denote the exterior algebra of $L$ over
$S$. The complex $\Hom_S(\Lambda_S^\bullet L,N)$ with Chevalley--Eilenberg
differentials computes~$H_S^\*(L,N)$. 
\end{Remark}


\begin{Theorem}[\cite{kola}]\label{thm:spectral}
Let $(S,L)$ be a Lie--Rinehart algebra with enveloping algebra~$U$, and suppose
that $L$ is an $S$-projective module.
There exist a $U$-module structure on
$H^\*(S,U)$ and a first-quadrant spectral sequence $E_\*$ converging to
$\HH^\bullet(U)$ with second page
\[
  E_2^{p,q}
	= H_S^p(L,H^q (S,U)).
\]
\end{Theorem}

\begin{Proposition}\label{prop:Diff:HH}
There are isomorphisms 
$\HH^\*(\Diff\B_{n+1})\cong\HH^\*(\Diff\tilde\B_n)$
for any $n\geq1$.
\end{Proposition}

\begin{proof} This is a consequence of applying the K\"unneth's formula for
Hochschild cohomology as in H.\,Cartan and
S.\,Eilenberg's~\cite{cartan-eilenberg}*{XI.3.I} to the isomorphism  
$ \Diff \B_{n+1} \cong A_1 \otimes \Diff\tilde\B_n$ in
Proposition~\ref{prop:braidshh} and the observation that $\HH^0(A_1)\cong \kk$
and $\HH^i(A_1)\cong\kk$ if $i\neq 0$.
\end{proof}

\begin{Corollary}\label{coro:hhB3}
The Hilbert series of the Hochschild cohomology of $\Diff\B_3$ is 
\[
  h(t)=1 + 3t + 6t^2 + 4t^3.
\]
\end{Corollary}

\begin{proof} Proposition~\ref{prop:Diff:HH} particularizes to
$\HH^\*(\Diff\B_3)\cong\HH^\*(\Diff\tilde\B_2)$, and then
\cite{kola}*{Corollary 5.8}  reads
\(
  h_{\HH^*(\Diff\B_3)} = h_{\HH^*(\Diff\tilde\B_2)} = 1 + 3t + 6t^2 + 4t^3.
\)
\end{proof}

\section{Combinatorics of the Koszul complex}

We let $n\geq1$ and assume throughout
this section that $(S,L)$ is a Lie--Rinehart algebra with
$S=\kk[x_1,\ldots,x_n]$ and 
$L$ a free $S$-module with basis $\left( \alpha_1,\ldots,\alpha_n
\right)$. Let $U=U(S,L)$ be its Lie--Rinehart enveloping algebra.
To compute the Hochschild cohomology of $S$ we use the Koszul resolution of
$S$ available in~\cite{weibel}*{\S 4.5}.

\begin{Lemma}\label{lem:koszul}
Let $W$ be the subspace of $S$ with basis $(x_1,\ldots,x_n)$.
The complex $P_\*=S^e\otimes\Lambda^\* W$
with differentials $b_\* :P_\*\to P_{\*-1}$
defined for $s, t\in S$, $q\in\nset{n}$ and $1\leq i_1<\dots<i_{q}\leq n$ by
\begin{align*}
  &b_q(s|t\otimes x_{i_1}\wedge \dots\wedge x_{i_q})
  = \sum_{j=1}^q(-1)^{j+1}
  [(sx_{i_j}|t) -(s|x_{i_j}t) ]\otimes x_{i_1}\wedge\dots\wedge\check
  x_{i_j}\wedge\dots\wedge x_{i_q}
\end{align*}
and augmentation $\varepsilon : S^e\to S $ given by $\varepsilon(s|t) = st$ is
a resolution of $S$ by free $S^e$-modules. The notation is the usual one: the
symbol $|$ denotes the tensor product inside $S^e$ and $\check x_{i_j}$ means
that $x_{i_j}$ is omitted.
\end{Lemma}

Through a classical
adjunction, the complex $\Hom_{S^e}(P_\*,U)$ is isomorphic to 
\[\label{eq:koszul-adjunction}
  \Hom(\Lambda^\*W,U)
  \cong 
  U\otimes \Hom(\Lambda^\*W,\kk)
  \eqqcolon
  \X^\*.
\] 
We compute the Hochschild cohomology $H^\bullet(S,U)$ from the complex 
$(\X^\*,d^\*)$. For each $q$ in $\nset{0,n}$ the basis 
\(
  \{ \xx_{k_1}\wedge\ldots\wedge\xx_{k_q} : 1\leq k_1<\ldots<k_q\leq n\} 
\)
of  $\Hom(\Lambda^qW,\kk)$
dual to the basis $\{x_{k_1}\wedge\ldots\wedge x_{k_q}\}$ of
$\Lambda^qW$ 
induces a basis of  $\X^q$ as a $U$-module.


Write $\alpha^I\coloneqq \alpha_{n}^{i_{n}}\ldots\alpha_{1}^{i_1}$ for each
$n$-tuple of nonnegative integers $I=(i_{n},\dots,i_{1})$, and call $\abs{I} =
i_{n}+\ldots+i_{1}$ the \emph{order} of $I$.  A result \textit{à la}
Poincaré-Birkhoff-Witt in~\cite{rinehart}*{\S3} assures that the set 
\[\label{eq:pbw}
  \left\{ \alpha^I : I\in\NN^n\right\}
\]
is an $S$-basis of $U$. Moreover, $U$ is a
filtered algebra, with filtration $(F_pU:p\geq0)$ given by \emph{the order of
differential operators}: 
\(
  F_pU
  = \langle 
      f\alpha^I
      : f\in S, \abs I\leq p 
    \rangle
\)
for each~$p\geq0$.

\begin{Proposition}\label{prop:X-PBW}
Let $q\in\{0,\ldots,n\}$.
\begin{thmlist}
\item The set formed by
\(
    \alpha^I\xx_{k_1}\wedge\cdots\wedge\xx_{k_q}
\)
with  $I\in\NN^n$ and $1\leq k_1<\ldots<k_q\leq n$
is an $S$-basis of $\X^q$.

\item There is a filtration
$(F_p\X^q:p\geq0)$ of vector spaces
on $\X^q$ determined for each $p\geq0$ by
\[\label{eq:filtU}
  F_p\X^q
  = \langle 
      f\alpha^I\xx_{k_1}\wedge\cdots\wedge\xx_{k_q} 
      : f\in S, 1\leq k_1<\ldots<k_q\leq n, 
      I\in\NN^n \text{ such that } \abs I\leq p 
    \rangle.
\]
\end{thmlist}
\end{Proposition}

\begin{proof}
In view of \eqref{eq:koszul-adjunction}, 
for each $q$ the $U$-module $\X^q$ admits
$
  \left\{ 
    \xx_{k_1}\wedge\cdots\wedge\xx_{k_q} 
    : 1\leq k_1<\ldots<k_q\leq n
  \right\}
$
as a basis.
The claim follows from this and the $S$-basis of $U$ in~\eqref{eq:pbw} above.
\end{proof}

The differentials
$d^q:\X^q\to\X^{q+1}$ induced by $b_\* :P_\*\to P_{\*-1}$
satisfy for $q=0,1$ 
\begin{align*}
  &
  d^0 : u \mapsto \sum_{k=1}^n[u,x_k]\xx_k,
  &&
  d^1 :
  \sum_{k=1}^nu_k\xx_k
  \mapsto  
  \sum_{1\leq k<l\leq n}  \left( [u_k,x_l]-[u_l,x_k] \right)\xx_k\wedge\xx_l.
\end{align*}
Given $m\in\nset{n}$ we denote by $e_m$ the $n$-tuple whose components are
all zero except for the~$(n-m)$th, where there is a~$1$.

\begin{Lemma}\label{lem:axkalphai}
Let $a = \sum_{\abs{I}=p}f^I\alpha^I$ for $f^I\in S$
with $I\in\NN^n$.
If $k\in\nset{n}$ and $J=(j_{n},\ldots,j_{1})\in\NN^n$ has order $p-1$
then the component of $[a,x_k]$ in $\alpha^J$ is 
\[
  \sum_{m=1}^{n} (j_m+1)\alpha_m(x_k)f^{J+e_m}.
\]
\end{Lemma}

\begin{proof}
If $I=(i_{n},\ldots,i_{1})\in\NN^n$ has order $p$ then
\begin{align*}
  [f^I\alpha^I,x_k]
  &\equiv
  i_{n}f^I\alpha_{n}(x_k)\alpha^{I-e_{n}}
  + \ldots
  + i_{1}f^{I}\alpha_{1}(x_k)\alpha^{I-e_{1}}
  \mod F_{p-2}U.
\end{align*}
If there exists a monomial in this expression belonging to $S\alpha^J$ then
there exists $m\in\nset{n}$ such that $I-e_m=J$. This happens when the 
component of $[f^I\alpha^I,x_k]$ in $\alpha^J$ is 
\[
  i_m\alpha_m(x_k)f^{I} 
  = (j_m+1)\alpha_m(x_k)f^{J+e_m},
\]
and therefore
\(
  [a,x_k]
  \equiv
  \sum_{\abs J = p-1}
    \sum_{m=1}^{n} (j_m+1)\alpha_m(x_k)f^{J+e_m}
  \alpha^J
\)
modulo $F_{p-2}U$.
\end{proof}

For $g^1,\ldots,g^n\in S $ and
$f^1=(f^1_1,\ldots,f^1_n),\ldots,f^n=(f^n_1,\ldots,f^n_n)\in S^{\times n }$ we
let
\begin{align}\label{eq:OmegaF1H0}
  &
  \Omega^0(g^{n},\ldots,g^{1})
  \coloneqq \sum_{i=1}^{n}g^i\alpha_i \in F_1U,
  &&
  \Omega^1(f^{n},\ldots,f^{1})
  \coloneqq \sum_{l=1}^n\sum_{i=1}^{n}f_k^i\alpha_i\xx_k \in F_1\X^1.
\end{align}

\begin{Proposition}\label{prop:differentials}
Let $p\geq0$, $u\in F_pU$ and $\omega\in F_p\X^1$. 
\begin{thmlist}
\item\label{prop:differentials0}
If $\{f^I : I\in \NN^n, \abs I = p\}\subset S $ is such that
$
  u
  \equiv \sum_{\abs{I} =p}
    f^I\alpha^I
  \mod F_{p-1}U
$~then
\[
  d^0(u)
  \equiv \sum_{\abs{J}=p-1} 
    d^0 \left( \Omega^0\left( 
    (j_{n}+1)f^{J+e_{n}},(j_{n-1}+1)f^{J+e_{n-1}},\ldots,(j_{1}+1)f^{J+e_{1}} 
    \right)  \right)  \alpha^J
\]
modulo $F_{p-2}\X^1$.
\item\label{prop:differentials1}
If 
$
  \omega 
  \equiv  \sum_{l=1}^n\sum_{\abs{I} = p}
    f_l^I\alpha^I \xx_i
  \mod F_{p-1}\X^1
$
for
$\left\{f^I_l : I\in\NN^n , \abs I = p, l\in\nset{n}\right\}\subset S $
then
\[
  d^1(\omega)
  \equiv \sum_{\abs{J}=p-1} 
    d^1 \left( \Omega^1\left( 
      (j_{n}+1)f^{J+e_{n}},(j_{n-1}+1)f^{J+e_{n-1}},\ldots,(j_{1}+1)f^{J+e_{1}} 
    \right)  \right)  \alpha^J  
\]
modulo $F_{p-2}\X^2$.
\end{thmlist}
\end{Proposition}

\begin{proof}
To prove~\ref{prop:differentials0} it suffices to see that 
the desired equality holds in each coefficient of the $S$-basis
$\left( \alpha^J\xx_k : J\in\NN^n, k\in\nset{n} \right)$ 
of $\X^1$ given in Proposition~\ref{prop:X-PBW}.
Let then $J=(j_{n},\ldots,j_{1})\in\NN^n$ of order $p-1$ and $k\in\nset{n}$. 
Thanks to Lemma~\ref{lem:axkalphai} the component in
$\alpha^J\xx_k$
of $d^0(u)=\sum_{l=1}^n[u,x_l]\xx_l$~is 
\[\label{eq:differentials0-omegap}
  \sum_{m=1}^{n} (j_m+1)\alpha_m(x_k)f^{J+e_m}.
\]
On the other hand, given $f^{n},\ldots,f^{1}\in S$ a direct calculation shows
that the component in
$\xx_k$ of $d^0\left( \Omega^0\left(f^{n},\ldots,f^{1}\right) \right)$
is 
\(
  \sum_{i={1}}^{n}f^i\alpha_i(x_k).
\)
It follows that the component of 
\[
  d^0\left( 
    \Omega^0\left( 
      (j_{n} + 1)f^{J+e_{n}}, \ldots, (j_{1} +1)f^{J+e_{1}}
    \right) 
  \right) 
\]
in $\xx_k$ is equal to~\eqref{eq:differentials0-omegap}, which is
tantamount to what we wanted to see. 
The proof of~\ref{prop:differentials1} is
completely analogous.
\end{proof}

\section{Cohomologies in degree zero and centers} 
\label{sec:h0}

In this section $(S,L)$ is a triangularizable Lie algebra:
$S=\kk[x_1,\ldots,x_n]$ for some $n\geq1$ and $L$ is a sub-$S$-module  of
derivations of $S$ with a basis given by derivations
$\alpha_1,\ldots,\alpha_n$ that satisfy $\alpha_i(x_j) = 0$ if $i>j$ and
$\alpha_i(x_i)\neq0$ for every $i\in\nset{n}$.  Let $U$ be the enveloping
algebra of $(S,L)$.

\subsection{The cohomology of $S$ with values on $U$}
The Hochschild cohomology $H^\*(S,U)$ is the cohomology of the
complex~$(\X^\*,d^\*)$ of~\eqref{eq:koszul-adjunction}.

\begin{Lemma}\label{lem:h0key}
The restriction of $d^0:\X^0\to\X^1$ to $F_1\X^0$ has kernel $F_0\X^0$.
\end{Lemma}

\begin{proof}
It is evident that $F_0\X^0=S$ is contained in~$\ker d^0$. Let $u\in F_1U$
and  $f^{1},\ldots,f^{n}\in S$ such that 
$ u\equiv \sum_{i=1}^{n}f^i\alpha_i$ modulo $S$.  We examine the equations
$d^0(u)(1|x_l|1)=0$, that is, $[u,x_l]=0$ for each $1\leq l\leq n$.  We first
observe that
\[
  0 
  = [u,x_1]
  = \sum_{i=1}^{n}f^i\alpha_i(x_1)
  = f^{1}\alpha_{1}(x_1),
\]
and then $f^1=0$.
Proceeding inductively on $k$, we assume that $ u= \sum_{i = k
}^{n}f^i\alpha_i$ and compute 
\[
  0 
  = [u,x_k]
  = f^{k}\alpha_{k}(x_k) + \ldots + f^{n}\alpha_{n}(x_k)
  = f^k\alpha_k(x_k).
\]
We deduce that $f^{k}=0$ and conclude that $u\in S$.
\end{proof}

\begin{Proposition}\label{prop:h0key}
If $p>0$ and $u\in F_pU$ are such that $d^0(u) \equiv 0$ modulo $F_{p-2}\X^1$
then $u\in F_{p-1}U$. 
\end{Proposition}

\begin{proof}
Let $\{ f^I : I\in\NN^n, \abs I = p \}\subset  S$ be such that  $u\equiv
\sum_{\abs I=p} f^I\alpha^I$ modulo $F_{p-1}U$: thanks to
Proposition~\ref{prop:differentials} we have that 
\[
  d^0(u)
  \equiv \sum_{\abs{J=(j_{n},\ldots,j_{1})}=p-1} 
    d \left( \Omega^0\left( 
      (j_{n}+1)f^{J+e_{n}},\ldots,(j_{1}+1)f^{J+e_{1}} 
    \right)  \right)  \alpha^J
  \mod F_{p-2}\X^1.
\]
We deduce that 
\(   
  0
  =
  d^0 \left( \Omega^0\left( 
      (j_{n}+1)f^{J+e_{n}},\ldots,(j_{1}+1)f^{J+e_{1}} 
  \right)  \right)
\)
for each $J$ with $\abs J = p-1$, 
provided that $p-1\geq0$.  Thanks to Lemma~\ref{lem:h0key} we deduce that $0
=f^{J+e_{N-2}}=\ldots=f^{J+e_{-1}}$. Since we can write every $I\in\NN$ with
$\abs I = p$ as $I = J +  e_m$ for some $m\in\nset{n}$ we conclude
that if $ p-1\geq0$ then $f^I = 0$ for every $I$ with $\abs I =p$.
\end{proof}

\begin{Proposition}\label{prop:h0}
The inclusion $S\inc U=\X^0$ induces an isomorphism of graded $U$-modules
$H^0(S,U) = S$
\end{Proposition}

\begin{proof}
Let us write $u = u_0+\ldots+ u_p$ with $u_q\in F_{q}U\setminus F_{q-1}U$
and $p\geq0$ maximal among those $q$ such that $u_q\neq 0$.
As $d^0(u) = 0$ and $d^0(u_q)\in F_{q-1}\X^1$ for every $q\in\nset{0,p}$
we have that $d(u_p) \equiv 0
\mod F_{p-2}X^1$, and we may use Proposition~\ref{prop:h0key} to see that if
$p>0$ then $u_p=0$.
We conclude then that $p=0$, so that actually $u\in S$. We obtain the 
result with the evident observation that  every element of $S$ is a
$0$-cocycle in $\X^\*$.
\end{proof}

\subsection{The cohomology of~$U$}
Our recent calculation of $H^0(S,U)$ leaves us just one step away from
the zeroth Hochschild cohomology space of~$U$. 

\begin{Theorem}\label{thm:hh0}
Let $(S,L)$ be a triangularizable Lie--Rinehart algebra with enveloping
algebra~$U$. There is an isomorphism of vector spaces $\HH^0(U)\cong \kk$.
\end{Theorem}

\begin{proof}
As a consequence of the immediate degeneracy of the spectral sequence of
Theorem~\ref{thm:spectral} there is an isomorphism of vector spaces
$\HH^0(U) \cong H_S^0(L,H^0(S,U))$. In view of
Proposition~\ref{prop:h0}, this isomorphism amounts to
\[
  \HH^0(U)
  \cong H_S^0(L,S)
  \cong \{ f\in S : \alpha_i(f) = 0\text{ if $i\in\nset{n}$}\}
\]
Since $\alpha_i(f)=\sum_{j=1}^n \alpha_i(x_j)\partial_jf$ for $f\in S$, the
condition that $\alpha_i(f) = 0$ if $i\in\nset{n}$ means that
$(\partial_1f,\ldots,\partial_nf)$ belongs to the kernel of the Saito's matrix
$M=\left( \alpha_i(x_j) \right)_{i,j=1}^n$. As this matrix is triangular and
its determinant is nonzero, the condition  $\alpha_i(f)=0$ for all
$i\in\nset{n}$ is equivalent to $\partial_jf=0$ for all $j\in\nset{n}$, which
is to say that $f\in\kk$.
\end{proof}

\begin{Corollary}
Let $r,n\geq1$. The centers of $\Diff\left( \A(C_r\wr\SS_n) \right)$ and 
of $\Diff\B_n$ are~$\kk$.
\end{Corollary}

\begin{proof}
The algebras considered have been shown to satisfy the hypotheses of
Theorem~\ref{thm:hh0} in Examples~\ref{ex:braids:2conditions}
and~\ref{ex:wreath}.
\end{proof}

\section{The first cohomology space \texorpdfstring{$H^1(S,U)$}{H1}}
\label{sec:H1SU}

We now restrict our attention to the case in which $n=3$. Let then $(S,L)$ be
a Lie-Rinehart algebra with $S=\kk[x_1,x_2,x_3]$ and $L$ the free $S$-module
generated by the subset of derivations $\{\alpha_1,\alpha_2,\alpha_3\}$ in
$\Der S$. We suppose that $(S,L)$ is triangularizable, this is, $\alpha_i(x_j)
= 0$ if $i>j$ and $\alpha_1(x_1)\alpha_2(x_2)\alpha_3(x_3)\neq0$, and that
$(S,L)$ satisfies the Bézout condition:
\begin{itemize}
\item the polynomials $\alpha_2(x_2)$ and $\alpha_2(x_3)$ are coprime;
\item the polynomials $\alpha_1(x_1)$ and 
\(
  \det
    \begin{psmallmatrix}
      \alpha_1(x_2) & \alpha_1(x_3) \\
      \alpha_2(x_2) & \alpha_2(x_2)
    \end{psmallmatrix}
\)
are coprime.
\end{itemize}

\begin{Lemma}\label{lem:A3:H1SU:F1}    
Let $\left\{f_l^i: i\in\left\{ 1,2 \right\}, l\in\nset{3}\right\}\subset S$
and write $\omega=\sum_{l=1}^3\left( f_l^2\alpha_2 + f_l^1\alpha_1
\right)\xx_l\in\X^1$. If $\omega$  is a cocycle then there exist unique elements
$g_{11},g_{12},g_{22}$ of $S$ such that $g_{11}\alpha_1(x_1) = f_1^1$,
$g_{12}\alpha_1(x_1)=f_1^2$ and
$g_{22}\alpha_2(x_2) = f_2^2-g_{12}\alpha_1(x_2)$. These elements satisfy
\[
  \omega  
  \equiv d(\tfrac{1}{2}g_{11}\alpha_1^2 + g_{12}\alpha_2\alpha_1
    +\tfrac{1}{2}g_{22}\alpha_2^2 )
  \mod F_0\X^1.
\]
\end{Lemma}

\begin{proof}
The components in $\xx_1\wedge\xx_2$ and $\xx_1\wedge\xx_3$ of $d\omega=0$
tell us that $\alpha_1(x_j)f_1^1 +\alpha_2(x_j)f_1^j =\alpha_1(x_1)f_j^1 $ 
for $j\in\{2,3\}$. We can arrange these two equations as 
\[\label{eq:exA3:1-j}
  \begin{pmatrix}
    \alpha_1(x_2) & \alpha_2(x_2)
    \\
    \alpha_1(x_3) & \alpha_2(x_3)
  \end{pmatrix}
  \begin{pmatrix}
    f_1^{1} \\ f_1^{2} 
  \end{pmatrix}
  = 
  \alpha_1(x_1 )
  \begin{pmatrix}
     f_{2}^{1} \\  f_{3}^{1}
  \end{pmatrix}
\]
and then Cramer's rule tells us that if $i\in\{1,2\}$ then
\(
  f_1^i\det \tilde M = \alpha_1(x_1) \det \tilde M_i,
\)  
where $\tilde M$ is the matrix on the left hand of \eqref{eq:exA3:1-j} and $\tilde M_i$ is
the matrix obtained by replacing the $i$th column of $\tilde M$ by
$
  \begin{psmallmatrix}
     f_{2}^{1} \\  f_{3}^{1}
  \end{psmallmatrix}.
$
It follows that $\alpha_1(x_1)$ divides $f_1^i$ ---because it is coprime with
$\det \tilde M$ in view of the Bézout hypothesis--- and then there exist
$g_{11}$ and $g_{12}$ in $S$ such that $g_{1i}\alpha_1(x_1)=f_1^i$.
 
Let $u_1\coloneqq \tfrac{1}{2}g_{11}\alpha_1^2 + g_{12}\alpha_2\alpha_1$
and
$\tilde\omega \coloneqq \omega - d(u_1)$, and write
$\tilde\omega=\sum_{l=1}^3\left( \tilde f_l^2\alpha_2 + \tilde f_l^1\alpha_1
\right)\xx_l$.~Since 
\begin{align}
  d(u_1):x_1 & \mapsto [u,x_1]
  \equiv
  g_{11}\alpha(x_1)\alpha_1 + g_{12}\alpha_1(x_1)\alpha_2 
  \mod S
  \\
  &\hphantom{\mapsto [u,x_1]}
  \;= f_1^2\alpha_1 + f_1^2\alpha_2 ,  
  \\
  x_2 & \mapsto [u,x_2]
  \equiv
   g_{11}\alpha_1(x_2)\alpha_1 
    + g_{12}\alpha_2(x_2)\alpha_1
    + g_{12}\alpha_1(x_2)\alpha_2 
  \mod S
  \label{eq:exA3:du1:2}
\end{align}
we have $\tilde f_1^1=\tilde f_1^2=0$. Now,
the equation $d\tilde\omega=0$ in $\xx_1\wedge\xx_2$ and $\xx_1\wedge\xx_3$
tells us, as in~\eqref{eq:exA3:1-j}, that $\tilde f_2^1=\tilde f_3^1=0$,
and in $\xx_2\wedge\xx_3$ that 
$ \alpha_2(x_2)\tilde f_3^2 = \alpha_2(x_3)\tilde f_2^2 $. Thanks to
the Bézout condition there exists
$g_{22}\in S$ such that $g_{22}\alpha_2(x_2) = \tilde f_2^2$; in view
of~\eqref{eq:exA3:du1:2}, $\tilde f_2^2$ is equal to
$f_2^2-g_{12}\alpha_1(x_2)$.
Put $u_2 \coloneqq \tfrac{1}{2} g_{22}\alpha_2^2$. We see that
$d(u_2)(x_1)=0$ and that 
\begin{align}
\MoveEqLeft
  d(u_2)(x_2)
  =[u_2,x_2] 
  \equiv g_{22}\alpha_2(x_2) \alpha_2 
  \mod S
  \\
  &= \tilde f_2^2\alpha_2.
\end{align}
The difference $\bar\omega\coloneqq\tilde\omega-d(u_2)$ is therefore a
coboundary with no component modulo $S$ in $\xx_1$ nor in $\xx_2$, so we can
write $\bar\omega \equiv \left( f^1\alpha_1 + f^2\alpha_2  \right)\xx_3\mod
F_0\X^1$.  Now, the equations that come from $\bar\omega$ being a coboundary
are $0=f^1\alpha_1(x_1)$ in $\xx_1\wedge\xx_3$, from which $f^1=0$, and
$0=f^2\alpha_2(x_2)$ in $\xx_2\wedge\xx_3$, whence finally $\bar\omega\in
F_0\X^1$.  We have in this way obtained that $\omega \equiv d(u_1+u_2)\mod
F_0\X^1$, as desired.
\end{proof}

\begin{Proposition}\label{prop:A3:H1SU:Fp}
Let $p\geq0$ and $\omega\in F_p\X^1$, and let $\left\{ f^{(i_3,i_2,i_1)}_l :
l\in\nset{3}, i_3,i_2,i_1\geq0 \right\}\subset S $ such~that
\[\label{eq:exA3:omega:p}
  \omega 
  \equiv \sum_{l=1}^3\sum_{i_1+i_2+i_3= p}
    f_l^{(i_3,i_2,i_1)}\alpha^{(i_3,i_2,i_1)} \xx_l
  \mod F_{p-1}\X^1.
\]
If $\omega$ is a cocycle and
$f^{(p,0,0)}_1=f^{(p,0,0)}_2=f^{(p,0,0)}_3=0$ then $\omega\in F_{p-1}\X^1$.
\end{Proposition}

\begin{proof}
Let us prove by descending induction on $i$ from $p$ to $0$ that 
\[\label{claim:exA3:H1SU:Fp:induct}
  \claim{the cocycle $\omega$ is cohomologous modulo $F_{p-1}\X^1$ 
    to a cocycle of the form~\eqref{eq:exA3:omega:p}    
    with $f^{(i_{3},i_2,i_{1})}_l=0$  if $l\in\nset{3}$ and $i_{3} \geq i$.
  }
\]
Our hypotheses give us the truth of~\eqref{claim:exA3:H1SU:Fp:induct} for
$i=p$.
Suppose now that~\eqref{claim:exA3:H1SU:Fp:induct} is true for $p,\ldots,i$
and assume, without loss of generality, that $\omega$ \emph{is} of the
form~\eqref{eq:exA3:omega:p} with $f^{(i_{3},i_2,i_{1})}_l=0$ if
$l\in\nset{3}$, $i_1+i_2+i_3= p$ and~$i_{3} \geq i$.

\begin{Lemma}\label{lem:A3:H1SU:Fp:induct2}
Let $q\in\{0,\dots,p-i+1\}$. The cocycle $\omega$ is cohomologous modulo
$F_{p-1}\X^1$ to a cocycle of the form~\eqref{eq:exA3:omega:p}    with
$f^{(i-1,p-i+1,0)}_l=\ldots=f^{(i-1,p-i+1-q,q)}_l=0$ and
$f^{(i_{3},i_2,i_{1})}_l=0$  if $i_1+i_2+i_3= p$ and $i_{3} \geq i$ for every
$l\in\nset{3}$.
\end{Lemma}

The auxiliary result above implies at once the truth of the inductive
step of the proof of~\eqref{claim:exA3:H1SU:Fp:induct}, 
thus demonstrating Proposition~\ref{prop:A3:H1SU:Fp}.
\end{proof}

\begin{proof}[Proof of Lemma~\ref{lem:A3:H1SU:Fp:induct2}]
Suppose that $q=0$. Equation $d\omega=0$
in its component $(i-1,p-i,0)$ reads, thanks
to~Proposition~\ref{prop:differentials}, 
\[
  0 
  = d\left(      
      \Omega^1(if^{(i,p-i,0)},(p-i+1)f^{(i-1,p-i+1,0)},f^{(i-1,p-i,1)}) 
    \right)
\]
and the inductive hypothesis~\eqref{claim:exA3:H1SU:Fp:induct} tells us that
$f^{(i,p-i,0)}=0$. Applying now Lemma~\ref{lem:A3:H1SU:F1} in
we obtain that there are 
$g_{11},g_{12},g_{22} \in S $ such that 
\begin{align}
  & g_{11}\alpha_1(x_1)  = f^{(i-1,p-i,1)}_1,
  \qquad g_{12}\alpha_1(x_1) =(p-i+1)f^{(i-1,p-i+1,0)}_1,
  \\
  & g_{22}\alpha_2(x_2) = (p-i+1)f^{(i-1,p-i+1,0)}_2-g_{12}\alpha_1(x_2)
\end{align}

Let $v = (\tfrac{1}{2} g_{11}\alpha_1^2 + \tfrac{1}{(p-i+1)}
g_{12}\alpha_2\alpha_1)\alpha^{(i-1,p-i,0)}$ and write
$\tilde\omega=\omega-d(v)$, so that there exists $\left\{ \tilde
f_l^{(i_3,i_2,i_1)} \right\}\subset S$ such that
\[
  \tilde\omega 
  \equiv \sum_{l=3}^n\sum_{i_1+i_2+i_3= p}
    \tilde f_l^{(i_3,i_2,i_1)}\alpha^{(i_3,i_2,i_1)} \xx_l
  \mod F_{p-1}\X^1.
\]
Recall that $d(v):x_l \mapsto [v,x_l]$ for $l\in\nset{3}$. Since
\begin{align}
\MoveEqLeft[2]
  [v,x_1]
  \equiv
  (g_{11}\alpha_1(x_1)\alpha_1 +
\tfrac{1}{(p-i+1)}g_{12}\alpha_1(x_1)\alpha_2)\alpha^{(i-1,p-i,0)}
  \mod F_{p-1}U
  \\
  &
  =   f^{(i-1,p-i,1)}_1\alpha^{(i-1,p-i,1)} 
    + f^{(i-1,p-i+1,0)}_1\alpha^{(i-1,p-i+1,0)},
\end{align}
we have that $\tilde f^{(i-1,p-i,1)}_1 =\tilde f^{(i-1,p-i+1,0)}_1=0$.
Moreover, as $[v,x_2],[v,x_3]\in \bigoplus_{ i_3< i}S\alpha^{(i_3,i_2,i_1)}$
the coefficients $\tilde f_l^{(i_3,i_2,i_1)}$ are equal to
$f_l^{(i_3,i_2,i_1)}$ and therefore to zero if $i_1+i_2+i_3= p$,
$i_3\geq i$ and $l\in\nset{3}$.

We now look at equation $d\tilde\omega=0$, again in its coefficient of
$\alpha^{(i-1,p-i,0)}$ to obtain that 
\(
  0 
  = d\left(      
      \Omega^1(0,(p-i+1)\tilde f^{(i-1,p-i+1,0)},\tilde f^{(i-1,p-i,1)}) 
    \right).
\)
This equation in its component in $\xx_1\wedge\xx_2$ tells us, thanks
to~\eqref{eq:exA3:1-j}, that $\tilde f^{(i-1,p-i,1)}_2=\tilde
f^{(i-1,p-i,1)}_3=0$. On the other hand, applying~Lemma~\ref{lem:A3:H1SU:F1}
we get
%
$g\in S$ such that
$g\alpha_2(x_2) = (p-i+1)\tilde f^{(i-1,p-i+1,0)}_2$. 
Let now $\lambda = 1/(p-i+2)(p-i+1)$ and $\tilde v =\lambda g
\alpha^{(i-1,p-i+2,0)}$. Since $[\tilde v,x_1]=0$, 
\begin{align}
\MoveEqLeft[2]
  [\tilde v,x_2]
  \equiv
  \lambda g (p-i+2)\alpha_2(x_2)\alpha^{(i-1,p-i+1,0)}
\mod F_{p-1}U
  \\
  & = \lambda  (p-i+1)\tilde f^{(i-1,p-i+1,0)}_2(p-i+2)\alpha^{(i-1,p-i+1,0)}
  \\
  & =\tilde f^{(i-1,p-i+1,0)}_2\alpha^{(i-1,p-i+1,0)}
\end{align}
and $[\tilde v,x_3] \in \bigoplus_{ i_3<i}S\alpha^{(i_3,i_2,i_1)}$, the
difference $\tilde\omega - d(\tilde v)$ is a cohomologous modulo $F_{p-1}\X^1$
to a cocycle 
$\eta 
  = \sum_{l=3}^n\sum_{i_1+i_2+i_3= p}
    h_l^{(i_3,i_2,i_1)}\alpha^{(i_3,i_2,i_1)} \xx_l$
with $h^{(i_3,i_2,i_1)}_l=0$ if $i_1+i_2+i_3=p$, $i_3\geq i $ and
$l\in\nset{3}$ and $h^{(i-1,p-i,1)}_l = h^{(i-1,p-i+1,0)}_l=0$ if
$l\in\{1,2\}$.  Applying~Lemma~\ref{lem:A3:H1SU:F1} one final time we obtain
$\tilde g_{11},\tilde g_{12},\tilde g_{22} \in S $ such that $\tilde
g_{11}\alpha_1(x_1) =(p-i+1) h^{(i-1,p-i+1,0)}_1$, $\tilde
g_{12}\alpha_1(x_1)=h^{(i-1,p-i,1)}_1$ and $\tilde g_{22}\alpha_2(x_2) =
(p-i+1)h^{(i-1,p-i,1)}_2-\tilde g_{12}\alpha_1(x_2)$ ---and therefore $\tilde
g_{11}$, $\tilde g_{12}$ and $\tilde g_{22}$ must be equal to~$0$---   and
that satisfy
\[
  \Omega^1(0,(p-i+1)h^{(i-1,p-i+1,0)},h^{(i-1,p-i,1)}) 
  \equiv 
    d(\tfrac{1}{2}\tilde g_{11}\alpha_1^2 + \tilde g_{12}\alpha_2\alpha_1
    + \tfrac{1}{2}\tilde g_{22}\alpha_2^2 )
  \mod F_0\X^1.
\]
It follows that $h^{(i-1,p-i+1,0)}=h^{(i-1,p-i,1)}=0$, and therefore
that $\eta\equiv 0\mod F_{p-1}\X^1$. This finishes the proof
of the base step of~Lemma~\ref{lem:A3:H1SU:Fp:induct2}.

\bigskip

We finally deal with the inductive step of~Lemma~\ref{lem:A3:H1SU:Fp:induct2}. 
Let $q$, $i$ and $\omega$ be as in the statement. The component in
$\alpha^{(i-1,p-i-q,q)}$ of equation $d\omega =0$ yields
\[
  0 = 
  d\left( \Omega^1\left( 
    if^{(i,p-i-q,q)},(p-i-q+1)f^{(i-1,p-i-q+1,q)},(q+1)f^{(i-1,p-i-q,q+1)}
  \right) \right).
\]
Now, our inductive hypotheses of~\eqref{claim:exA3:H1SU:Fp:induct} and of
Lemma~\ref{lem:A3:H1SU:Fp:induct2} tell us, respectively, that
$f^{(i,p-i-q,q)}=0$ and that $f^{(i-1,p-i-q+1,q)}=0$, and therefore our
equation above reduces to 
\[
  0 = 
  d\left( \Omega^1\left( 
    0,0,f^{(i-1,p-i-q,q+1)}
  \right) \right).
\]
Applying to this situation~Lemma~\ref{lem:A3:H1SU:F1} we obtain $g\in S$
such that $g\alpha_1(x_l) = f^{(i-1,p-i-q,q+1)}_l$ for $l\in\nset{3}$.  Let
$v=\tfrac{1}{(q+2)} g\alpha^{(i-1,p-i-q,q+2)}$ and write
$\tilde\omega=\omega-d(v)$: let
 $\left\{ f_l^{(i_3,i_2,i_1)} \right\}\subset S$ such that
\(
  \tilde\omega 
  \equiv \sum_{l=3}^n\sum_{i_1+i_2+i_3= p}
    \tilde f_l^{(i_3,i_2,i_1)}\alpha^{(i_3,i_2,i_1)} \xx_l
\)
modulo $ F_{p-1}\X^1$. As
\begin{align*}
  &[v,x_1] 
  \equiv g\alpha_1(x_1)\alpha^{(i-1,p-i-q,q+1)}
  = f^{(i-1,p-i-q,q+1)}_1\alpha^{(i-1,p-i-q,q+1)}
  \mod F_{p-1}U
\intertext{and if $j\in\left\{ 2,3 \right\}$ then }
  &[v,x_j]
  \in S\alpha^{(i-2,p-i-q,q+2)}
    \oplus S\alpha^{(i-1,p-i-q-1,q+2)}
    \oplus S\alpha^{(i-1,p-i-q,q+1)}
    \oplus F_{p-1}U,
\end{align*}
we obtain that  $\tilde f_l^{(i_3,i_2,i_1)}=0$ whenever
$i_3\geq i$ and 
$\tilde f^{(i-1,p-i+1,0)}_l=\ldots=\tilde f^{(i-1,p-i+1-q,q)}_l=0$ 
for every $l\in\nset{3}$ and, in addition, that
$\tilde f^{(i-1,p-i-q,q+1)}_1=0$. As a consequence of this, 
the component in
$\alpha^{(i-1,p-i-q,q)}$ of equation $d\tilde\omega =0$ 
reduces to 
\(
  0 = 
  d\left( \Omega^1\left( 
    0,0,\tilde f^{(i-1,p-i-q,q+1)}
  \right) \right).
\)
The element $\tilde g$ that is provided for this situation
by~Lemma~\ref{lem:A3:H1SU:F1} satisfies $\tilde g\alpha_1(x_l) =\tilde
f^{(i-1,p-i-q,q+1)}_l$ for $l\in\nset{3}$: it follows that $g=0$ and hence
$\tilde f^{(i-1,p-i-q,q+1)}_l=0$ for $l\in\nset{3}$ and $\tilde\omega\equiv
0\mod F_{p-1}\X^1$. This finishes the proof of
Lemma~\ref{lem:A3:H1SU:Fp:induct2}.
\end{proof}

From this point on we demand to $(S,L)$ that in addition it satisfy the
orthogonality condition: that there be a family $u_1,u_2,u_3$ of elements of
$U$ that can be written as $u_k = \alpha_3 + h_k^2\alpha_2 + h_k^1\alpha_1$
for some $\left\{ h_k^i : k\in\nset{3}, i\in\nset{2}\right\}\subset S$ and
such that $[u_k,x_l]=0$ whenever $k\neq l$.  The idea is that that we can add
to any cocycle in $F_p\X^1$ an $S$-linear combination of $u_k^p\xx_k$ to 
remove its components in the maximum power of $\alpha_3$ and in this way
obtain a cocycle that falls in the hypotheses of
Proposition~\ref{prop:A3:H1SU:Fp}. 

\begin{Corollary}\label{coro:A3:H1generators}
Let $\{u_1,u_2,u_3\}$ be the family that makes $(S,L)$ satisfy the
orthogonality condition.
\begin{thmlist}
\item The cochains $\eta^p_k=u_k^p\xx_k\in F_p\X^1$ defined for~$p\geq0$
and~$k\in\nset{3}$ are cocycles.

\item\label{coro:A3:H1generators:these} 
Every cocycle in $\X^1$ is cohomologous to one in the $S$-submodule
of $\X^1$ generated by  $\left\{ \eta_k^p  : k\in\nset{3}, p\geq0 \right\}$.
\end{thmlist}
\end{Corollary}

\begin{proof}
Let us denote by $Z^1$ the $S$-module generated by $\left\{ \eta_l^p  :
l\in\nset{3}, p\geq0 \right\}$.  We prove by induction on $p\geq0$ that if
$\omega\in F_p\X^1$ is a cocycle then there exist $z\in Z^1$ and $u\in U$ 
such that $ \omega  = d^0(u) + z$.  We first observe that $F_0(\X^1)=F_0(Z^1)$
because $\xx_l = \eta_l^0$, and then for $p=0$ we have that
$\omega\in F_0(\X^1)\subset Z^1$.

Assume now that  $p>0$ and let $\left\{ f^I_l : l\in\nset{3}, I\in\NN^3
\right\}\subset S $ such that
\(
  \omega 
  \equiv \sum_{l=1}^3\sum_{\abs{I} = p}
    f_l^I\alpha^I \xx_l
  \mod F_{p-1}\X^1.
\)
Defining $z = \sum_{l=1}^3
f_l^{(p,0,0)}\eta_l^p$ we see that the cocycle
\(
  \tilde\omega \coloneqq \omega - z
\)
has its components in $\alpha^{p}\xx_1,\alpha^{p}\xx_2,\alpha^{p}\xx_3$ equal to
zero, and applying Proposition~\ref{prop:A3:H1SU:Fp} we deduce that
$\tilde\omega$ is a coboundary modulo $F_{p-1}\X^1$: let $u\in U$ and
$\omega'\in F_{p-1}\X^1$ be such that $\tilde\omega = d^0(u) +\omega'$.  The
inductive hypothesis tells us that there exist $u'\in U$ and $z'\in Z^1$ such
that $\omega' = d^0(u') + z'$, and thus $\omega = \tilde\omega + z = d^0(u+u')
+(z+z')$, as we wanted.
\end{proof}

\begin{Proposition}
Let $p\geq0$.
\begin{thmlist}
\item
Let $\omega\in F_p\X^1$ be a cocycle, so 
that there exist $\{f_1,f_2,f_3\}\subset S$ and $u\in U$ such that 
$\omega \equiv \sum_{l=1}^3f_l\eta^p_l + du$ modulo $F_{p-1}\X^1$.  
The cocycle $\omega$ is equivalent to a coboundary modulo $F_{p-1}\X^1$ if and
only if $\sum_{l=1}^3f_l\xx_l$ is a coboundary.

\item
The unique $S$-linear map $\gamma_p : F_0\X^1 \to  
F_pX^1$ such that $\xx_l\mapsto\eta_l^p$
if $1\leq l\leq 3$ induces an isomorphism of $S$-modules
\[
  F_{p}H^1(S,U) / F_{p-1}H^1(S,U) \cong F_0H^1(S,U).
\]
\end{thmlist}
\end{Proposition}

\begin{proof}
Suppose that $\omega_0 = \sum_{l=1}^nf_l\xx_l$ is a coboundary and let $v\in
U$ such that $d^0(v) = \omega_0$. Thanks to Proposition~\ref{prop:h0key} 
we may assume that
$v\in F_1\X^1$ and write 
$v\equiv g^{3}\alpha_3+g^2\alpha_2+g^1\alpha_1\mod S$ for some
$g^{3},g^2,g^{1}\in S$.
In view of Proposition~\ref{prop:differentials} there exist
$f_0^{I} \in F_0\X^1$ such that
we may write
\begin{align*}
\MoveEqLeft
  d^0\left( 
    \frac{g^{3}}{p+1}\alpha_{3}^{p+1} 
    + g^2\alpha_{3}^p\alpha_{2}
    + g^1\alpha_{3}^p\alpha_{1}
  \right)
  \equiv
  d^0\left(v \right)\alpha^p_{3}
  + \sum_{i_3<p} f_0^{(i_{3},i_2,i_{1})} \alpha^{(i_{3},i_2,i_{1})}
  \mod F_{p-1}\X^1.
\end{align*}
It follows that the difference 
\(
  \omega -
  d^0\left( 
    \frac{g^{3}}{p+1}\alpha_{3}^{p+1} 
    + g^2\alpha_{3}^p\alpha_{2}
    + g^1\alpha_{3}^p\alpha_{1}
  \right)
\)
is a cochain whose components in
$\alpha_{3}^p\xx_1,\alpha_3^p\xx_2,\alpha_{3}^p\xx_N$ are zero. Applying
Proposition~\ref{prop:A3:H1SU:Fp} we see that $\omega$ is equivalent to a
coboundary modulo $F_{p-1}\X^1$. 

Reciprocally, let $u\in U$ such that $d^0(u)= \omega$.  Thanks to
Proposition~\ref{prop:h0key} we know that $u\in F_{p+1}U$: let us write $u
\equiv \sum_{\abs K = p+1}h^{K} \alpha^K$ with $\{h^{K}: K\in\NN^3, \abs K =
p+1\}\subset S$.
Taking into account Proposition~\ref{prop:differentials} again we see that
\[
  d^0(u)
  \equiv 
  d\left( \Omega^0\left( 
    (p+1)h^{(p+1,0,0)},h^{(p,1,0)},h^{(p,0,1)}
  \right) \right) \alpha^p_{3}
  + \sum_{i_{3}<p}  f_0^{(i_{3},i_2,i_{1})} \alpha^{(i_{3},i_2,i_{1})}
\]
modulo $F_{p-1}\X^1$ for some
$f_0^{I} \in F_0\X^1$. The equality of this to $\omega$ implies, looking at the
components in $\alpha_{3}^p\xx_1,\alpha_3^p\xx_2,\alpha_{3}^p\xx_3$, that
\(
  d\left( \Omega^0\left( 
    (p+1)h^{(p+1,0,0)},h^{(p,1,0)},h^{(p,0,1)}
  \right) \right)  
  = \sum_{l=1}^3 f_l\xx_l.
\)
This completes the proof of the first item.

Now, the truth of the first item implies two things: first, that the composition
$F_0\X^1\to F_p\X^1/F_{p-1}\X^1$ of $\gamma_p$ with the projection to the
quotient descends to cohomology and, second,
that the map induced in cohomology by this composition is a monomorphism. It
is also surjective thanks to
Corollary~\ref{coro:A3:H1generators}\ref{coro:A3:H1generators:these}.
\end{proof}

Recall that the filtered $S$-module $F_\*H^1(S,U)$ has a graded associated
$S$-module $\Gr_\*H^1(S,U) = \bigoplus_{p\geq0} \Gr_pH^1(S,U)$ given by
$\Gr_pH^1(S,U) \coloneqq F_pH^1(S,U)/F_{p-1}H^1(S,U)$. We have just
seen that $\Gr_pH^1(S,U)$ is isomorphic as an $S$-module to
$F_0H^1(S,U)$ for any $p\geq0$: we claim that we can make it an isomorphism 
of \emph{graded} $S$-modules.

Given $p\geq0$, the map $\gamma_p : F_0\X^1 \to  F_pX^1$ induces an
isomorphism of $S$-modules $F_0H^1(S,U)\cong \Gr_pH^1(S,U)$ that shifts the
polynomial degree in $3(p-1)$: indeed, for each $l\in\nset{3}$
the class of $\eta_l\in\X^1$, which has
polynomial degree $2$, is sent to the class of $\eta_l^p$, which has
polynomial degree $3p-1$.
On the other hand, the morphism of $S$-modules $\gamma:F_0(H^1(S,U))\otimes
\kk[\alpha_3]\to\Gr H^1(S,U)$ such that $[\eta_l]\otimes \alpha_3^p \mapsto 
[\eta_l^p]$ for $l\in\nset{3}$ and $p\geq0$ does respect the
graduation and is an isomorphism because so is each $\gamma_p$.
In addition to this, we observe that
\[
  F_0(H^1(S,U)) 
  = \frac{S\otimes\lin{\xx_1,\xx_2,\xx_3}}
    {Sd^0(\alpha_{1})+Sd^0(\alpha_2)+Sd^0(\alpha_{3})}
  \cong \coker M.
\]
We summarize our findings in the following statement.

\begin{Corollary}\label{coro:H1:result:graded}
Let 
$S=\kk[x_1,x_2,x_3]$ and  $L$ a free $S$-submodule 
of $\Der S$ generated by derivations $\alpha_1,\alpha_2,\alpha_3$
in such a way that  $(S,L)$ is a triangularizable Lie--Rinehart algebra that
satisfies the Bézout and orthogonality conditions.
Let~$U$ be the Lie–Rinehart enveloping algebra of~$L$.
There is an isomorphism of $S$-graded modules
\begin{align*}
  H^1(S,U)\cong \coker M \otimes \kk[\alpha_3],
\end{align*}
where $M$ is the Saito's matrix of $(S,L)$.
\end{Corollary}

Recall that the cokernel of $M$ has a rich algebraic structure ---
see M.\,Granger, D.\,Mond and M.\,Schulze's~\cite{schulze}.

%
\section{Computation of $\HH^1(U)$}
\label{sec:hh1}

The spectral sequence in Theorem~\ref{thm:spectral}, regardless of its
degeneracy, gives us an strategy to obtain the first Hochschild cohomology
space $\HH^\*(U)$ of the enveloping algebra $U$ of a Lie--Rinehart algebra
$(S,L)$: indeed, $\HH^1(U)$ is isomophic to
$H^1_S(L,H^0(S,U))\oplus  H^0_S(L,H^1(S,U))$. 
In Sections~\ref{sec:h0} and~\ref{sec:H1SU} 
we computed $H^0(S,U)$
and $H^1(S,U)$ when $(S,L)$ is as in 
Corollary~\ref{coro:H1:result:graded}, which is
the conclusion of Section~\ref{sec:H1SU}. 
In this section we describe their $L$-module structure and use
it to compute their respective Lie--Rinehart cohomology spaces
for the case in which $(S,L)$ is associated to a hyperplane arrangement
of the form $\A_r=\A(C_r\wr\SS_3)$ as in Example~\ref{ex:Ar}.
%
%
\subsection{The $L$-module structure on \texorpdfstring{$H^\*(S,U)$}{H(S,U)}}
\label{subsec:H1SU:actionofL}

Let $(S,L)$ be a Lie-Rinehart pair with enveloping algebra $U$.
Let us  describe the construction in~\cite{kola} that gives
 an $L$-module structure to the Hochschild cohomology $H^\*(S,U)$ of
$S$ with values on $U$.

Fix $\alpha\in L$ and an $S^e$-projective resolution $\varepsilon:P_\*\to S$.
Let $\alpha_\*$ be an $\alpha^e_S$-lifting of
$\alpha_S:S\to S $ to~$P_\*$, that is, a morphism of complexes
$\alpha_\*=(\alpha_{q}:P_q\to P_q)_{q\geq0}$ such that
$\varepsilon\circ\alpha_0= \alpha_S\circ \varepsilon$ and for each $q\geq0$,
$s$, $t\in S$ and $p\in P_q$
\[
  \alpha_q( ( s\otimes t)\cdot p)
  =\left( \alpha_S(s)\otimes t + s\otimes\alpha_S(t) \right)\cdot p 
  +(s\otimes t)\cdot p.
\]
The endomorphism $\alpha_\*^\sharp$ of $\Hom_{S^e}(P_\*,U)$, defined for each
$q\geq0$  to be
\begin{align}\label{eq:alphasharp}
  \alpha^\sharp_q (\phi): p \mapsto
  [\alpha,\phi(p)] - \phi\circ\alpha_q (p)
  \quad\text{whenever $\phi\in\Hom_{S^e}(P_q,U)$ and $p\in P_q$,}
\end{align}
allows us to define the map~$\nabla_\alpha^\*:H^\*(S,U)\to H^\*(S,U)$ as
the unique graded endomorphism such that 
\[\label{eq:alphanabla}
  \nabla_\alpha^q([\phi]) = [\alpha_q^\sharp(\phi)],
\]
where [-] denotes class in cohomology. The final result is that 
$\alpha\mapsto \nabla^q_\alpha$ defines an $L$-module structure on~$H^q(S,U)$
for each $q\geq0$.

\subsection{The liftings}
From now on we work on the Lie--Rinehart algebra associated to $\A_r$ and 
put $E\coloneqq \alpha_1$, $D\coloneqq \alpha_2$
and $C\coloneqq \alpha_3$.
The commuting relations in $L$ are determined by the rules
\[\label{eq:A_r:commuting}
\begin{aligned}
  & [E,C] = (2r+1)C,
  &&[E,D] = (r+1)D,
  \\
  & [D,C] = r(x_3^r+x_2^r-x_1^r),
\end{aligned}
\]
as a straightforward calculation shows.

\begin{Proposition}\label{prop:liftings}
The rules 
\begin{align}\label{eq:Dlifting}
  & D_1(1|x_1|1) = 0, \\
  & D_1(1|x_k|1) 
  = \sum_{s+t=r}x_k^s|x_k|x_k^t 
    -\sum_{s+t=r-1}x_k^s|x_1|x_1^tx_k
    - x_1^r|x_k|1
  \quad\text{if $k=2,3$}  
\end{align}
define a $D^e$-lifting of $D:S\to S$.
\end{Proposition}

\begin{proof}
It is evident that $d_1\circ D_1$ and $D_0\circ d_1$ coincide at $1|x_1|1$; if
$k=2,3$ then 
$d_1\circ D_1(1|x_k|1) $ is
\[
\!\begin{multlined}[.85\displaywidth] 
  \sum_{s+t=r}x_k^s(x_k|1-1|x_k)x_k^t
    -\sum_{s+t=r-1}x_k^s(x_1|1 - 1 |x_1) x_1^tx_k
    - x_1^rx_k|1 + x_1^r|x_k
  \\
  = \left( x_k^{r+1} - x_kx_1^r \right)|1 
    - 1|\left( x_k^{r+1} - x_kx_1^r \right),
\end{multlined}
\]
which equals $D_0\circ d_1(1|x_k|1) = D_0(x_k|1-1|x_k) = D(x_k)|1-1|D(x_k)$
because $x_k(x_k^r-x_1^r)$ is~$D(x_k)$.
\end{proof}

\begin{Proposition}\label{prop:Dsharp1}
For every $p\geq0$ we have that 
\begin{align}
\label{eq:Dsharp1}
  &D_1^\sharp(\eta_1^p)
  \equiv
  pr(x_3^r+x_2^r-x_1^r)\eta_1^p
    + rx_1^{r-1}x_2\eta_2^p 
    + rx_1^{r-1}x_3\eta_3^p\mod F_{p-1}\X^1 \mod \im d^0
  \\
  &D_1^\sharp ( \eta_2^p)
  \equiv \left( (1-p)x_1^r +(p-r-1)x_2^r + px_3^r  \right) \eta_2^p
  \mod F_{p-1}\X^1, 
  \\  
  & D_1^\sharp ( \eta_3^p)
  \equiv \left( (1-p)x_1^r+px_2^r + (p-r-1)x_3^r  \right) \eta_3^p
  \mod F_{p-1}\X^1. 
\end{align}
\end{Proposition}

\begin{proof}
Recall from Corollary~\ref{coro:A3:H1generators}
the cocycle $\eta_l^p=u_l^p\xx_l$
for each  $l\in\nset{3}$ that is such that $u_l$ commutes with every
$x_j$ with $j\neq l$. The commuting relations~\eqref{eq:A_r:commuting}
in $L$ give 
\begin{align}
  [D,u_3] 
  &= [D,C] = r(x_3^r+x_2^r-x_1^r)C
  = r(x_3^r+x_2^r-x_1^r)u_3
\shortintertext{and}
  [D,u_2] 
  &= [D,C-(x_3^r-x_2^r)D] 
  \\
  &
  = r(x_3^r+x_2^r-x_1^r)C 
    - \left(rx_3^{r}(x_3^r-x_1^r) - rx_2^{r}(x_2^r-x_1^r)\right)D
  \\
  &= r(x_3^r+x_2^r-x_1^r)\left( C-(x_3^r-x_2^r)D \right) 
  =r(x_3^r+x_2^r-x_1^r)u_2.
\end{align}
It follows that if $l=2,3$
then $[D,u_l^p]\equiv p(x_3^r+x_2^r-x_1^r)u_l^p$ modulo $F_{p-1}U$. We can now
compute
\begin{align*}
\MoveEqLeft
  D_1^\sharp(\eta_l^p)(1|x_1|1)
  = [D,\eta_l^p(1|x_1|1)] -\eta_l^p\circ D_1(1|x_1|1)  
  = [D,0] - \eta_l^p(0)  = 0  ;
  \\
\MoveEqLeft
  D_1^\sharp(\eta_l^p)(1|x_l|1)
  = [D,\eta_l^p(1|x_l|1)] -\eta_l^p\circ D_1(1|x_l|1) 
  \\
  &= 
    [D,u_l^p] - \eta_l^p
      \left( 
        \sum_{s+t=r}x_l^s|x_l|x_l^t 
        -\sum_{s+t=r-1}x_l^s|x_1|x_1^tx_l
        - x_1^r|x_l|1 
      \right) 
  \\
  &\equiv  
    p(x_3^r+x_2^r-x_1^r)u_l^p
      - \left( (r+1)x_l^r - x_1^r  \right) u_l^p
  \mod F_{p-1}U
\intertext{and for $m\neq 1,l$}
\MoveEqLeft
   D_1^\sharp(\eta_l^p)(1|x_m|1)
  = [D,\eta_l^p(1|x_m|1)] - \eta_l^p\circ D_1(1|x_m|1) 
  \\
  &= [D,0]  - \eta_l^p
      \left( 
        \sum_{s+t=r}x_m^s|x_m|x_m^t 
        -\sum_{s+t=r-1}x_1^s|x_1|x_1^tx_m
        - x_1^r|x_m|1 
      \right)   
  = 0.
\end{align*}
With this information at hand we are able to see that
\begin{align*}
  &D_1^\sharp ( \eta_2^p)
  \equiv \left( (1-p)x_1^r +(p-r-1)x_2^r + px_3^r  \right) \eta_2^p
  \mod F_{p-1}\X^1, 
  \\  
  & D_1^\sharp ( \eta_3^p)
  \equiv \left( (1-p)x_1^r+px_2^r + (p-r-1)x_3^r  \right) \eta_3^p
  \mod F_{p-1}\X^1. 
\end{align*}
Let us now consider the action of $D$ on $\eta_1^p$. To begin with, we have
\begin{align}
\MoveEqLeft[2]  
  [D,u_1^p] 
  \equiv pu_1^{p-1}[D,u_1] \mod F_{p-1}U 
  \\
  &= pu_1^{p-1}[D,(C- (x_3^r-x_1^r)D + (x_3^r-x_1^r)(x_2^r-x_1^r)E)] 
  \\
  &\!\begin{multlined}[.8\displaywidth]
  = pu_1^{p-1}
    \Big( 
      r(x_3^r+x_2^r-x_1^r)C 
      - r x_3^r(x_3^r-x_1^r)D 
      \\
      + D\left( (x_3^r-x_1^r)(x_2^r-x_1^r) \right)E 
      - (r+1)(x_3^r-x_1^r)(x_2^r-x_1^r)D
    \Big)
  \end{multlined}
\end{align}
and we observe that 
\(
  D_1^\sharp(\eta_1^p)(1|x_1|1)
  = [D,u_1^p] -\eta_1^p(D_1(1|x_1|1))  
  = [D,u_1^p].
\)
On the other hand, if $m\in\{2,3\}$ then 
\begin{align*}
\MoveEqLeft
  D_1^\sharp(\eta_1^p)(1|x_m|1)
  = [D,\eta_1^p(1|x_m|1)] -\eta_1^p(D_1(1|x_m|1))  
  \\
  &= [D,0] - \eta_1^p
      \left( 
        \sum_{s+t=r}x_m^s|x_m|x_m^t 
        - \sum_{s+t=r-1}x_1^s|x_1|x_1^tx_m
        - x_1^r|x_m|1 
      \right)
  \\
  &\equiv
    rx_1^{r-1}x_mu_1^p.    
\end{align*}
From these computations we
see that the cocycle 
\[
  D_1^\sharp(\eta_1^p) 
  - pr(x_3^r+x_2^r-x_1^r)\eta_1^p
  - rx_1^{r-1}x_2\eta_2^p 
  - rx_1^{r-1}x_3\eta_3^p 
\]
has component zero in $C^p\xx_1$, $C^p\xx_2$ and $C^p\xx_3$, and then
Proposition~\ref{prop:A3:H1SU:Fp} tells us that
$ D_1^\sharp(\eta_1^p)$ is cohomologous modulo $ F_{p-1}\X^1$ to
$pr(x_3^r+x_2^r-x_1^r)\eta_1^p
    + rx_1^{r-1}x_2\eta_2^p 
    + rx_1^{r-1}x_3\eta_3^p$.
\end{proof}

\subsection{Invariants of $H^1(S,U)$ by the action of $L$}

We already have explicit descriptions of $H^1(S,U)$, in Section~\ref{sec:H1SU},
and of
the action of $L$ thereon, in Subsection~\ref{subsec:H1SU:actionofL} above:
the next step is to calculate the intersection of the kernels of the actions
of $E$, $D$ and $C$ on $H^1(S,U)$.

\begin{Proposition}\label{prop:H0(H1)}
$H^0_S(L,H^1(S,U))=0$.
\end{Proposition}

\begin{proof}
Recall that the polynomial grading on $S$ induces a grading on $S$, on $U$ and
on the cohomology $H^\*(S,U)$.
Since the derivation $E$ induces the linear endomorphism $\nabla_E^1$ of
$H^1(S,U)$ that sends the class of an homogeneous element $a$ of degree
$\abs{a}$ to the class of $\abs{a}a$ it follows that
\(
  \ker \nabla^1_E = H^1(S,U)_0,
\)
where $H^1(S,U)_0$ is the subspace of $ H^1(S,U)$ formed by elements of degree
zero.
Remember that if $k\in\nset{3}$ then
$\abs{u_k} = \abs{\alpha_3} = 2r$, and therefore $\abs{\eta_k} = 2r-1$. 
In view of our calculation in Corollary~\ref{coro:H1:result:graded}
this means that
\[
  \ker \nabla^1_E 
  = H^1(S,U)_0
  \cong
  \begin{cases*}
    \frac{ S_1\otimes\lin{\xx_1,\xx_2,\xx_3} }{\kk(x_1\xx+x_2\xx_2+x_3\xx_3) }
      \oplus \lin{\eta_1,\eta_2,\eta_3}
    & if $r=1$;
    \\
    \frac{ S_1\otimes\lin{\xx_1,\xx_2,\xx_3} }{\kk(x_1\xx+x_2\xx_2+x_3\xx_3) }
    & if $r\geq2$.
  \end{cases*}
\]

We begin by supposing that $r\geq2$. We observe that if $f_1,f_2,f_3\in S_1$
then
\begin{align*}
\MoveEqLeft[2]
  D_1^\sharp \left(  \sum f_i\xx_i  \right)
  = \sum \left( D(f_i)\xx_i + f_iD_1^\sharp(\xx_i) \right)
  \\
  &\!\begin{multlined}[.8\displaywidth] 
    = \sum  D(f_i)\xx_i
      + f_1r\left(  x_1^{r-1}x_2\xx_2 +  x_1^{r-1}x_3\xx_3 \right)
      + f_2 (x_1^r-(r+1)x_2^r)\xx_2
      \\
      + f_3 (x_1^r-(r+1)x_3^r)\xx_3
  \end{multlined}
  \\
  &\!\begin{multlined}[.8\displaywidth] 
    =  D(f_1) \xx_1
      + \left( D(f_2) + f_1rx_1^{r-1}x_2 + f_2(x_1^r-(r+1)x_2^r)\right)\xx_2
    \\
      + \left( D(f_3) + f_1rx_1^{r-1}x_3 + f_3(x_1^r-(r+1)x_3^r)\right)\xx_3 
  \end{multlined}
\end{align*}
belongs to the homogeneous component of degree $r$ of 
$F_0H^1(S,U)$, which is precisely 
\[
  \left( 
    \frac{S\otimes\lin{\xx_1,\xx_2,\xx_3}}
      {Sd^0(E)+Sd^0(D)+Sd^0(C)} 
  \right)_r
  = \frac{S_{r+1}\otimes\lin{\xx_1,\xx_2,\xx_3}}
    {S_rd^0(E)+\kk d^0(D)}.
\]
It follows that if
\(
  \nabla_D^1\left( [ \sum f_i\xx_i ] \right) 
  = \left[D_1^\sharp \left(  \sum f_i\xx_i  \right)\right]
\)
is zero in cohomology there must exist $g\in S_r$ and $\mu\in\kk$ such that
\[\label{eq:H1SU-noDetas}
  D_1^\sharp \left(  \sum f_i\xx_i  \right)
  = g(x_1\xx_1 + x_2\xx_2+ x_3\xx_3 )
    + \mu\left( x_2(x_2^r-x_1^r)\xx_2 - x_3(x_3^r-x_1^r)\xx_3 \right).
\]
Let us write $f_i = f_{i,1}x_1 + f_{i,2}x_2 + f_{i,3}x_3$ with $f_{i,j}\in\kk$
for $i,j\in\nset{3}$.
Up to the addition of coboundary that is a scalar multiple of 
$d^0(E) = x_1\xx_1+x_2\xx_2+x_3\xx_3$ we may suppose that $f_{1,1}=0$.
In $\xx_1$ we have $D(f_1) =gx_1$, or, in other words, 
\[
  f_{1,2}(x_2^{r+1}-x_1x_2^r) + f_{1,3}(x_3^{r+1}-x_1x_3^r)
  = gx_1.
\]
The components in $x_2^{r+1} $ and in $x_3^{r+1}$ of this equality read
$f_{1,2}= 0$ and $f_{1,3}=0$: this implies that $g=0$, and of course that
$f_1=0$.
Next, equation~\eqref{eq:H1SU-noDetas} in $\xx_2$ yields the equality in
$S_{r+1} $
\[
  D(f_2) + f_2(x_1^r-(r+1)x_2^r)
  =\mu( x_2(x_2^r-x_1^r)).
\]
In $x_1^{r+1}$ and $x_3^{r+1}$ we have $f_{2,1}=0$ and $f_{2,3}=0$,
and what remains is 
$ -rf_{2,2}x_2^{r+1}=\mu( x_2(x_2^r-x_1^r)$. It follows that $\mu=0$
and therefore $f_2=0$; analogously, $f_3=0$.
We conclude that $\ker\nabla_D^1\vert_{H^1(S,U)_0}=0$ when $r\geq2$.

\bigskip

Let us now suppose that $r=1$ and compute the kernel of the restriction of
$\nabla^1_D$ to $ H^1(S,U)_0 $. Let then 
$f_1,f_2,f_3\in S $  and $\lambda_1,\lambda_2,\lambda_3$ be such that
$\nabla_D^1\left( [ \sum f_i\xx_i + \lambda_i\eta_i] \right)$ 
is zero in cohomology. Since 
\[
\!\begin{multlined}[.85\displaywidth] 
  H^1(S,U)_1
  \cong 
  \frac{ S_2\otimes\lin{\xx_1,\xx_2,\xx_3} }
    {S_1(x_1\xx+x_2\xx_2+x_3\xx_3) + \kk(x_2(x_2-x_1)\xx_2+x_3(x_3-x_1)\xx_3)}
  \\
  \oplus
  \frac{ S_1 \lin{\eta_1,\eta_2,\eta_3}}{ \kk(x_1\eta_1+x_2\eta_2 + x_3\eta_3)
}
  \oplus \lin{\eta_1^2,\eta_2^2,\eta_3^2}
\end{multlined}
\]
there exist 
$\mu_1,\mu_2\in\kk$ and  $g\in S_1$ such that
\[\label{eq:H1SU-Detas}
\!\begin{multlined}[.9\displaywidth] 
  D_1^\sharp \left(  \sum f_i\xx_i + \lambda_i\eta_i   \right)
  = g(x_1\xx+x_2\xx_2+x_3\xx_3) 
    + \mu_2( x_2(x_2-x_1)\xx_2+x_3(x_3-x_1)\xx_3)
  \\ 
    + \mu_1(x_1\eta_1+x_2\eta_2 + x_3\eta_3)
\end{multlined}
\]
We know from Proposition~\ref{prop:Dsharp1} that modulo $S$
$D_1^\sharp(\eta_1) \equiv (x_3+x_2-x_1)\eta_1+x_2\eta_2+x_3\eta_3$,
$D_1^\sharp(\eta_2) \equiv (-x_2+x_3)\eta_2$ and
$D_1^\sharp(\eta_3) \equiv (x_2-x_3)\eta_3$
and since $D_1^\sharp(\sum S\xx_i)\subset S$ the
equality~\eqref{eq:H1SU-Detas} implies that
\[
\!\begin{multlined}[.9\displaywidth] 
  \lambda_1(x_3+x_2-x_1)\eta_1
  + \left( \lambda_1x_2 +  \lambda_2  (-x_2+x_3)\right)\eta_2
  + \left( \lambda_1 x_3 + \lambda_3 (x_2-x_3) \right)\eta_3
  \\
  = \mu_1(x_1\eta_1+x_2\eta_2 + x_3\eta_3)
\end{multlined}
\]
This is an equality in $\bigoplus_{i=1}^3S_1\eta_i$. In $S_1\eta_3$ we have
\(
  \lambda_3x_2 + (\lambda_1-\lambda_3)x_3 = \mu_1x_3,
\)
so $\lambda_3 = 0$ and $\lambda_1 = \mu_1$. In $S_1\eta_2$ an analogous
argument shows that $\lambda_2=0$ and $\lambda_1 = \mu_1$ again, and finally
in $S_1\eta_1$ we have
\[
   \lambda_1(x_3+x_2-x_1)=\lambda_1x_1.
\]
It follows that $\lambda_1=\mu_1=0$.
Consider now what is left of~\eqref{eq:H1SU-Detas}: it is
precisely~\eqref{eq:H1SU-noDetas} replacing $r$ by $1$. The same argument,
therefore, allows us to see that $\ker\nabla_D^1\vert_{H^1(S,U)_0}=0$ 
when $r=1$.

\bigskip

We conclude  that
\(
  H^0_S(L,H^0(S,L))\subset  
  \ker\left( \nabla_D^1 :H^1(S,U)_0\to H^1(S,U)_1 \right)
  = 0,
\)
from which $H^0_S(L,H^1(S,L))=0$ independently of $r\geq1$.
\end{proof}

\begin{Corollary}\label{coro:Ar:HH1}
Let $r\geq1$ and $A_r=\A(C_r\wr\SS_3)$. 
If $(S,L)$ is its associated Lie--Rinehart algebra and
$U$ its enveloping algebra then $\HH^1(U)\cong H^1_S(L,S)$. In particular, 
the dimension of $\HH^1(U)$ is $3r+3$, the number of hyperplanes of $\A_r$.
\end{Corollary}

\begin{proof}
Thanks to Theorem~\ref{thm:spectral} $\HH^1(U)\cong H^1_S(L,H^0(S,U))\oplus
H^0_S(L,H^1(S,U))$; Proposition~\ref{prop:h0} tells us that $H^0(S,U)=S$ and
Proposition~\ref{prop:H0(H1)} above that the second summand is zero.
\end{proof}

Let $f\in S_1$ be a linear form whose kernel is one of the hyperplanes in
$\A_3$.  It is a direct verification that there is a unique derivation
$\partial_f:U\to U$ such that 
\[
\begin{cases*}
  \partial_f(g) = 0 & if $g\in S$; \\
  \partial_f(\theta) = \theta(f)/f & if $\theta\in\Der\A_r$.
\end{cases*}
\]
Fix as well $\kk=\mathbb C$ and factorize the defining polynomial
$Q(\A_r) = x_1x_2x_3\prod_{1\leq i<j\leq 3}(x_j^r-x_i^r)$
 as
\[\label{eq:Ar:factorized}
  Q(\A_r)
  = x_1x_2x_3\prod_{j=0}^{r-1}(x_2-e^{2j\pi i/r}x_1 )
    (x_3-e^{2j\pi i/r}x_1) (x_3-e^{2j\pi i/r}x_2)
\]

\begin{Corollary}\label{coro:abelian}
The Lie algebra of outer derivations of $\Diff\A_r$ together with the
commutator is an abelian Lie algebra of dimension $3r+3$ generated
by the classes of the derivations $\partial_f$ with $f$ in a linear factor
of~\eqref{eq:Ar:factorized}.
\end{Corollary}

\begin{proof}
We claim that the classes of $\partial_f$, with $f$ one of the linear factors
in~\eqref{eq:Ar:factorized}, are linearly
independent in $\Out(U)$. Indeed, let $u\in U$ and $\lambda_f\in\kk$ be such
that
\[\label{eq:gerst}
  \sum\lambda_f\partial_f(v)=[u,v]
  \qquad\text{for every $v\in U$.}
\]
Evaluating~\eqref{eq:gerst} on each $v=g\in S$ we obtain that the left side
vanishes and therefore $u\in H^0(S,U)$, which is equal to $S$ in view of
Proposition~\ref{prop:h0}.  Write $u=\sum_{j\geq0}u_j$ with $u_j\in S_j$.
Evaluating now~\eqref{eq:gerst} on $E$ we obtain that
$\sum_{f\in\AA}\lambda_f=-\sum_{j\geq0}ju_j$.  In each homogeneous component
$S_j$ with $j\neq0$ we have $ju_j=0$ and therefore $u\in S_0=\kk$ and, when
$j=0$, $\sum_f\lambda_f=0$.

Evaluating the left hand side of~\eqref{eq:gerst} on $C$ gives
$\sum_f\lambda_f\partial_f(C)$. Now, if $\partial_f(C) = C(f)/f =
\partial_3(f)C(x_3)/f$  is nonzero then $\partial_3(f)\neq0$ and thus
$f$ is a factor of $C(x_3)$: let us, then, factor $C(x_3)$ by
$x_3$ and $f_{l,j}=x_3-e^{2j\pi i/r}x_l$ for $l=1,2$ and $j\in\nset{0,r-1}$,
and in this way reformulate the evaluation of~\eqref{eq:gerst} at $C$ 
as the nullity of
\[
  \sum_{f\in\AA}\partial_3(f)C(x_3)/f
  = \lambda_{x_3}(x_3^r-x_2^r)(x_3^r-x_1^r)
  + \sum_{l=1,2}\sum_{j=0}^{r-1}
    \lambda_{f_{l,j}}x_3(x_3^r-x_2^r)(x_3^r-x_1^r)/f_{l,j}.
\]
Fix now $l\in\nset{2}$ and $j\in\nset{0,r-1}$ and apply the morphism of
algebras $\epsilon_{l,j}:S\to\kk[x_1,x_2]$ that sends $x_3$ to
$e^{2k\pi i/r}x_l$:
since $\epsilon_{l,j}\left( (x_3^r-x_{l'}^r)/f_{j',l'} \right)=0$ whenever $l\neq l'$ and
$j\neq j'$ we obtain that
\[
  \epsilon_{l,j} : \sum_{f\in\AA}\partial_3(f)C(x_3)/f
  \mapsto
  \lambda_{f_{l,j}}x_3(x_3^r-x_2^r)(x_3^r-x_1^r)/f_{l,j}.
\]
As the expression at which we evaluated $\epsilon_{l,j}$ was zero, it  follows
that $\lambda_{f_{l,j}}=0$ and, immediately, that also $\lambda_{x_3}=0$.

We observe that the indexes that survive in the sum 
$\sum\lambda_f\partial_f$ are $x_1$, $x_2$ and $f_j=x_2-e^{2j\pi i/r}x_1$ with
$j\in\nset{0,r-1}$; evaluating at $D$ we obtain
\[
  \sum\lambda_f\partial_f (D)
  = \lambda_{x_2}(x_2^r-x_1^r) 
  + \sum_{j=0}^{r-1} \lambda_{f_j}x_2(x_2^r-x_1^r)/f_j.
\]
Reasoning as above we get that $\lambda_{x_2} = \lambda_{f_j}=0$ for every
$j$. Recalling now that $\sum_{f\in\AA}\lambda_f=0$ we see that
$\lambda_{x_1}=0$ as well. 

The classes of $\partial_f$ with $f$ a linear factor
in~\eqref{eq:Ar:factorized} span $\Out U$ because the dimension
of $\Out U\cong\HH^1(U)$ is, thanks to
Corollary~\ref{coro:Ar:HH1}, precisely $\abs\A$.  The composition
$\partial_f\circ\partial_g:U\to U$ is evidently equal to zero for any
$f,g\in\AA$, as a straightforward calculation shows, and therefore the Lie
algebra structure in $\Out U$~vanishes.
\end{proof}


\end{document}